\documentclass[a4paper,12pt]{article}
\usepackage[ngerman, american]{babel}
\usepackage[utf8]{inputenc}
\usepackage[a4paper, left=2.5cm, right=2.5cm, top=2.5cm, bottom=2.5cm]{geometry}
\usepackage{graphicx}
\usepackage{float}
\usepackage{color,psfrag}
\usepackage{amssymb}
\usepackage{amsmath}
\usepackage{amsthm}
\usepackage{dsfont}
\usepackage{mathrsfs}
\usepackage{color}
\usepackage{verbatim}
\usepackage[usenames, dvipsnames]{xcolor}
\usepackage{hyperref}

\renewcommand{\phi}{\varphi}
\renewcommand{\rho}{\varrho}
\renewcommand{\epsilon}{\varepsilon}

\newcommand{\cB}{\mathcal{B}}

\newcommand{\cE}{\mathcal{E}}

\newcommand{\cL}{\mathcal{L}}

\newcommand{\RR}{\mathbb{R}}
\newcommand{\NN}{\mathbb{N}}
\newcommand{\ZZ}{\mathbb{Z}}
\newcommand{\QQ}{\mathbb{Q}}

\newcommand{\dx}{{\mathrm d}}

\newcommand{\supp}{\mathrm{supp}}
\renewcommand{\mod}{\mathrm{\ mod\ }}

\newcommand{\law}{\cL}

\newcommand{\bR}{\mathbb{R}}
\newcommand{\bN}{\mathbb{N}}

\newcommand{\re}{\mathrm{e}}
\newcommand{\di}{\mathrm{d}}
\newcommand{\ri}{\mathrm{i}}
\newcommand{\one}{\mathbf{1}}

\definecolor{violet}{rgb}{0.5, 0.0, 1.0}

\newtheorem{theorem}{Theorem}[section]
\newtheorem{lemma}[theorem]{Lemma}
\newtheorem{corollary}[theorem]{Corollary}

\newtheorem{remark}[theorem]{Remark}
\newtheorem{proposition}[theorem]{Proposition}
\newtheorem{example}[theorem]{Example}

\numberwithin{equation}{section}

\begin{document}
\date{February 1, 2021}
\author{Anita Behme\thanks{Technische Universit\"at Dresden, Institut f\"ur Mathematische Stochastik, 01062 Dresden, Germany, e-mail: anita.behme@tu-dresden.de}\,\,, Alexander Lindner,\, Jana Reker\thanks{Ulm University, Institute of Mathematical Finance, 89081 Ulm, Germany, e-mails: alexander.lindner@uni-ulm.de, jana.reker@uni-ulm.de}\,\, and Victor Rivero\thanks{Centro de Investigaci\'on Matem\'atics A.C. Calle Jalisco s/n Col. Mineral de Valenciana, C.P. 36240 Guanajuato, Guanajuato, M\'exico, e-mail: rivero@cimat.mx}}
\title{\vspace{-1cm}Continuity properties and the support of killed exponential functionals}
\maketitle
\begin{abstract}
For two independent L\'evy processes $\xi$ and $\eta$ and an exponentially distributed random variable $\tau$ with parameter $q>0$, independent of $\xi$ and $\eta$, the  killed exponential functional is given by $V_{q,\xi,\eta} := \int_0^\tau \re^{-\xi_{s-}} \, \di \eta_s$. Interpreting the case $q=0$ as~$\tau=\infty$, the random variable $V_{q,\xi,\eta}$ is a natural generalization of the exponential functional $\int_0^\infty \re^{-\xi_{s-}} \, \di \eta_s$, the law of which is well-studied in the literature as it is the stationary distribution of a generalised Ornstein--Uhlenbeck process. In this paper we show that also the law of the killed exponential functional $V_{q,\xi,\eta}$ arises as a stationary distribution of a solution to a stochastic differential equation, thus establishing a close connection to generalised Ornstein--Uhlenbeck processes. Moreover,
the support and continuity of the law of killed exponential functionals is characterised, and many sufficient conditions for absolute continuity are derived. We also obtain various new sufficient conditions for absolute continuity of $\smash{\int_0^te^{-\xi_{s-}}\dx\eta_s}$ for fixed $t\geq0$, as well as for  integrals of the form~$\smash{\int_0^\infty f(s) \, \dx\eta_s}$ for deterministic functions $f$. Furthermore, applying the same techniques to the case~$q=0$, new results on the absolute continuity  of the improper integral~$\int_0^\infty \re^{-\xi_{s-}} \, \di \eta_s$ are derived.
\end{abstract}



\section{Introduction}

Given two independent real-valued L\'evy processes $\xi$ and $\eta$, the \emph{generalised Ornstein--Uhlenbeck (GOU) process $(X_t)_{t\geq 0}$ driven by $\xi$ and $\eta$}
 is defined by
\begin{align} \label{eq-GOU}
X_t=\re^{-\xi_t}\Bigl(\int_0^t \re^{\xi_{s-}}d\eta_s+X_0\Bigr),\ t\geq0,
\end{align}
where $X_0$ is a starting random variable, independent of $\xi$ and $\eta$. This type of process has become a standard tool in stochastic modeling and thus appears in numerous applications such as finance and insurance, see, e.g., \cite{BNS01}, \cite{KLM2004}, or \cite{PAULSEN_RisksStochasticEnvironment1992} for a few examples. The GOU process can equivalently be  defined as the unique solution of the stochastic differential equation~(SDE) 
\begin{equation} \label{eq-GOUSDE} \di X_t = X_{t-} \, \di U_t + \di \eta_t, \quad t\geq 0,\end{equation} 
with starting value $X_0$, where $U$ is another L\'evy process, independent of $\eta$, and defined by the property that $\cE(U)_t = \re^{-\xi_t}$, where $\cE(U)$ denotes the Dol\'eans--Dade stochastic exponential of $U$, cf. \cite[p.428]{MallerMuellerSzimayer2009}. The GOU process is a Markov process, and it is known~(cf.~\cite[Thm. 2.1]{LindnerMaller2005}) that, provided $\xi$ and $\eta$ are not both deterministic,  it has an invariant probability distribution (i.e., stationary marginal distribution) if and only if the stochastic integral $\int_0^t \re^{-\xi_{s-}} \, \di \eta_s$  converges almost surely to a finite limit as~$t\to\infty$, in which case the limit random variable
\begin{align} \label{def-functional} V_{0,\xi,\eta} := \int_0^\infty \re^{-\xi_{s-}} \, \di \eta_s := \lim_{t\to \infty} \int_0^t \re^{-\xi_{s-}} \, \di \eta_s\end{align}
is called the \emph{exponential functional of $(\xi,\eta)$} and it has the same law as the invariant distribution. Necessary and sufficient conditions for convergence of the integral have been obtained by Erickson and Maller \cite{EricksonMaller2005}.
GOU processes and the law of $V_{0,\xi,\eta}$ are well studied in the literature, see, e.g., the survey paper~\cite{BertoinYor2005} or~\cite{BehmeLindner2015},~\cite{BehmeLindnerMaejima2016},~\cite{BertoinLindnerMaller2008},~\cite{CarmonaPetitYor1997},~\cite{deHaanKarandikar1989},~\cite{KuznetsovPardoSavov2012}, \cite{PardoRiveroSchaik}, to name just a few. In particular, in \cite[Thm. 1]{BehmeLindnerMaejima2016}, the support of $V_{0,\xi,\eta}$ has been characterised and shown to be always an interval, and in \cite[Thms. 2.2, 3.9]{BertoinLindnerMaller2008}, a characterisation of continuity of the law of $V_{0,\xi,\eta}$ and some sufficient conditions for its absolute continuity where derived.

\medskip
While we do add various new sufficient conditions for absolute continuity of the law of~$V_{0,\xi,\eta}$, the main focus of this paper is on killed exponential functionals. Let $\xi$ and $\eta$ be two independent L\'evy processes, $q\in (0,\infty)$ and $\tau$ an exponentially distributed random variable with parameter $q$, independent of $\xi$ and~$\eta$. Then
\begin{align}\label{def-killedfunctional}
V_{q,\xi,\eta} := \int_0^\tau e^{-\xi_{s-}} \, \di \eta_s
\end{align}
is called a {\it killed exponential functional of $(\xi,\eta)$ with parameter $q$}. Interpreting the case where $q$ is equal to zero as $\tau=\infty$, we obtain the classical case $\smash{\int_0^\infty e^{-\xi_{s-}} d\eta_s}$ as defined above and, unless the killing is explicitly specified, the term~"exponential functional" always refers to the improper integral \eqref{def-functional}. However, we may emphasize the difference by writing "exponential functional without killing" for $V_{0,\xi,\eta}$. Killed exponential functionals have so far mainly been considered when $\eta_t=t$, cf. \cite[Thm. 2]{Yor92}, \cite{PardoRiveroSchaik},~\cite{PatieSavov} or \cite{moehle}, to name just a few, where in \cite{PardoRiveroSchaik} absolute continuity of the law of $V_{q,\xi,t}$ is established, while in \cite{PatieSavov} the smoothness of the density is considered and the support of $V_{q,\xi,t}$ characterised. In \cite[Thm. 2]{Yor92} and \cite[p. 38]{moehle}, the distribution of $V_{q,\xi,t}$ is determined when $\xi$ is a Brownian motion or a particular L\'evy process.
Another important special case of the killed exponential functional is when $\xi_t = 0$, in which case $V_{q,0,\eta} = \eta_\tau$, which can be interpreted as the L\'evy process~$\eta$ subordinated by a gamma process with parameters 1 and $q>0$, evaluated at time 1. The law of $V_{q,0,\eta}$ is then~$q$ times the potential measure of $\eta$, cf. \cite[Def. 30.10]{Sato2013}.

\medskip
We start this paper recalling some preliminaries in Section \ref{S2}. Afterwards, to motivate our further studies on killed exponential functionals, we show in Section~\ref{S3} that, similar to the case without killing, also the law of $V_{q,\xi,\eta}$ appears naturally as the unique invariant probability measure of a certain Markov process that can be defined via an SDE of the form \eqref{eq-GOUSDE}. This proves the killed exponential functional $V_{q,\xi,\eta}$ to be a natural generalisation of the exponential function $V_{0,\xi,\eta}$ which appears whenever systems described via \eqref{eq-GOUSDE} are considered in a more general context. 
The resulting Markov process and some distributional properties of it have already been studied in \cite{BehmeLindnerMaller2011} and \cite{Behme2011}. We proceed in Section \ref{S4} to present a characterisation of the support of $V_{q,\xi,\eta}$ for $q>0$ and general L\'{e}vy processes~$\xi$ and $\eta$, which, unlike the cases $V_{0,\xi,\eta}$ or $V_{q,\xi,t}$, turns out not to be always an interval. Nevertheless, we find that the support of the law of the killed exponential functional is an interval in most cases and contains an unbounded interval under mild conditions on the jump structure of~$\xi$ and $\eta$ in the remaining cases.
In Section \ref{S5neu} we fully  characterise continuity of the law of $V_{q,\xi,\eta}$. Further, various sufficient conditions for absolute continuity of the law of~$V_{q,\xi,\eta}$ and of $V_{0,\xi,\eta}$ are obtained. Many of these conditions are new for the case without killing as well, so, e.g., we show (Theorem \ref{t-ac5}) that, as long as one of the two processes $\xi$ and $\eta$ is non-deterministic, $V_{0,\xi,\eta}$ is absolutely continuous whenever $\xi$ has a Gaussian component, or is of finite variation with non-zero drift, or if its L\'evy measure has an absolutely continuous component.
Since many of the mentioned results concerning continuity are derived by conditioning on the paths of~$\xi$ or on the killing time $\tau$, we treat integrals of the form $\smash{\int_0^\infty f(s) \, \di \eta_s}$, where $f$ is a deterministic function, as well as $\smash{\int_0^t \re^{-\xi_{s-}} \, \di \eta_s}$ for fixed $t\geq 0$ in Section \ref{S6new}, before being able to prove in Section \ref{S5newproofs} the results that were stated in Section \ref{S5neu}. The final Section \ref{S4proofs} contains the rather technical proofs of the results stated in Section \ref{S4}.


\section{Preliminaries} \label{S2}
A real-valued \emph{L\'evy process} $L=(L_t)_{t\geq0}$ is a stochastic process having stationary and independent increments, that starts in 0 and has almost surely (a.s.) c\`adl\`ag paths, i.e., paths that are right-continuous with finite left-limits. The distribution of a L\'evy process at time~$t\geq 0$ is infinitely divisible, and by the L\'evy-Khintchine formula its characteristic function is given by
$$E \re^{\ri z L_t} = \exp (t \Psi_L(z)) , \quad z\in \bR,$$
where $\Psi_L$ denotes the \emph{characteristic exponent} satisfying
\begin{align*}
\Psi_L(z)=\ri \gamma_Lz-\sigma_L^2z^2/2 +\int_{\RR}(\re^{\ri zx}-1-\ri zx\one_{[-1,1]}(x))\nu_L(\dx x), \quad z\in \bR.
\end{align*}
Here, $\sigma_L^2\geq 0$ is the \emph{Gaussian variance} of $L$, $\nu_L$ is the \emph{L\'evy measure} of $L$, $\gamma_L$ is the \emph{location parameter}, and $(\sigma_L^2,\nu_L,\gamma_L)$ is the \emph{characteristic triplet} of $L$. Note that in this paper we use the convention of defining $\nu_L$ as a measure on $\RR$ satisfying $\nu_L(\{0\})=0$. Whenever the process is of finite variation, i.e., when $\sigma_L^2=0$ and $\smash{\int_{[-1,1]} |x|\, \nu_L(\di x)} < \infty$, the \emph{drift}~${\gamma_L^0=\gamma_L-\smash{\int_{[-1,1]}x\, \nu_L(\dx x)}}$ is well-defined and we often use $\gamma_L^0$ instead of $\gamma_L$ to describe the distribution of $L_t$ in this case. A L\'evy process $L$ is \emph{spectrally positive} if $\nu_L((-\infty,0)) = 0$, and \emph{spectrally negative} if $\nu_L((0,\infty) = 0$. For further information on L\'{e}vy processes and infinitely divisible distributions we refer to \cite{Bertoin1996} or~\cite{Sato2013}.

\medskip
When considering the paths of $L$, we denote the space of real-valued c\`{a}dl\`{a}g functions on $[0,\infty)$ by $D([0,\infty),\RR)$. Further, for $f\in D([0,\infty),\RR)$, we set $f(s-)$ for the left-hand limit of $f$ at $s\in(0,\infty)$ and $\Delta f(s)=f(s)-f(s-)$ for the jumps.
Throughout the analysis, we often make use of certain subsets of the real numbers $\RR$, specifically~$\RR^*=\RR\setminus\{0\}$, $\RR_+=[0,\infty)$ and $\RR_-=(-\infty,0]$. The Borel $\sigma$-algebras on $\RR$ and $\RR^*$ are denoted by~$\cB_1$ and~$\cB_1^*$, respectively. Further, the symbol $\lambda^1$ is used for the (one-dimensional) Lebesgue measure on~$\cB_1$. The law of a random variable~$X$ is denoted by either~$\cL(X)$ or~$P_X$. When speaking of (absolute) continuity we will always mean (absolute) continuity with  respect to $\lambda^1$, and we say that a random variable~$X$ is (absolutely) continuous if its distribution $\law(X)$ has this property. The support of a measure $\mu$ on $(\bR,\cB_1)$ is denoted by $\supp(\mu)$, and by the support $\supp(X)$ of a random variable $X$  we mean the support of its distribution. Recall that $\supp(\mu)$ and $\supp(X)$ are by definition closed sets. If the random variable $X$ is exponentially distributed with parameter $q\geq0$, we set~$X\smash{\overset{d}{=}}{\rm{Exp}}(q)$ with the case $q=0$ being interpreted as $X\equiv\infty$.

\medskip
For integrals,  we assume the integral bounds to be included when using the notation~$\smash{\int_a^b}$ for $a,b\in\RR$ and indicate that the left bound is excluded by writing $\smash{\int_{a+}^b}$. Integrals of the form $\int_0^t \re^{-\xi_{s-}} \, \di \eta_s$ will be interpreted as integrals with respect to semimartingales as, e.g., in \cite{Protter2005}.
Since in this paper we shall always restrict to independent L\'evy processes $\xi$ and~$\eta$,
the stochastic (semimartingale)-integral $\smash{\int_0^t \re^{-\xi_{s-}} \, \di \eta_s}$ given the path $\xi=f$ is nothing else than the semimartingale-integral (equivalently, the Wiener-integral) $\int_0^t \re^{-f(s-)} \,\di \eta_s$ of the deterministic function $f$ with respect to $\eta$; this follows easily from the definition of the semimartingale-integral as in Protter \cite[Sect. II.4]{Protter2005}, e.g., when realising $\xi$ and~$\eta$ on a product space. The semimartingale-integral  $\smash{\int_0^t \re^{-f(s-)} \, \di \eta_s}$ then agrees with the corresponding stochastic integral in the sense of Rajput and Rosinski \cite[p. 460]{RajputRosinski1989} as both are limits in probability of integrals of simple functions, for which the corresponding integrals trivially agree. If a function $f:[0,t] \to \bR$ is integrable in the sense of Rajput and Rosinski (in particular this is satisfied when $f$ is bounded and measurable), then the distribution of the integral $\int_0^t f(s) \, \di \eta_s$ is infinitely divisible with characteristic exponent~$\Psi_{f,t}(z) = \smash{\int_0^t \Psi_\eta(f(s) z)\, \di s}$, $z\in \bR$, and
characteristic triplet $(\sigma_{f,t}^2,\nu_{f,t},\gamma_{f,t})$ given~by
\begin{align}
\sigma_{f,t}^2 & = \sigma_\eta^2 \int_0^t f(s)^2 \, \di s , \label{eq-LM7}\\
\nu_{f,t}(B) & =  \int_0^t \int_\bR \one_{B\setminus \{0\}} (f(s) x) \, \nu_\eta(\di x) \, \di s,  \quad B \in \cB_1, \quad \mbox{and}
\label{eq-LM1}\\
\gamma_{f,t} & =  \int_0^t \left[ f(s) \gamma_\eta + \int_\bR f(s) x \left( \one_{\{ |f(s)x| \leq 1\}} - \one_{\{|x|\leq 1\}} \right) \, \nu_\eta(\di x) \right] \, \di s ,\label{eq-LM5}
\end{align}
cf. \cite[Prop. 2.6, Thm. 2.7]{RajputRosinski1989} or \cite[Prop. 57.10]{Sato2013}.


\section{Killed exponential functionals as invariant distributions of Markov processes} \label{S3}

Throughout this section let $\xi$ and $\eta$ be two independent L\'evy processes.

\medskip
As mentioned in the introduction, the exponential functional $V_{0,\xi,\eta}$ without killing~(provided it converges) describes the stationary distribution of the GOU process \eqref{eq-GOU}, cf. \cite[Thm. 2.1]{LindnerMaller2005}. In this section we shall see that also the killed exponential functional $V_{q,\xi,\eta}$ arises as a stationary solution of a related Markov process.

\medskip
To formulate our result, recall that the stochastic exponential $\cE(L)$ of a L\'evy process~$L$ is the unique solution to the SDE $\dx\cE(L)=\cE(L)_{s-}\dx L_s$ with ${\cE(L)_0=1}$ a.s. It can be given explicitly through the Dol\'{e}ans-Dade formula (see, e.g.,~\cite[Thm. 37]{Protter2005}) according to which
$$\cE(L)_t = \exp\Bigl(L_t-\sigma_L^2 t/2 \Bigr)\prod_{0<u\leq t}(1+\Delta L_u)e^{-\Delta L_u}, \quad t\geq 0.$$

Thus defining
\begin{equation} \label{def-U}
U_t = - \xi_t + t \sigma_\xi^2/2 + \sum_{0<s\leq t} \left( \re^{-\Delta \xi_s} - 1 + \Delta \xi_s\right), \quad t\geq 0,
\end{equation}
it is readily checked that $U$ is a L\'evy process satisfying $\cE(U)_t = \re^{-\xi_t}$ and ${\nu_U((-\infty,-1]) = 0}$. Further, the GOU process $X$ of \eqref{eq-GOU} with starting random variable $X_0$, independent of~$\xi$ and $\eta$, is the unique solution of the SDE
\begin{equation*}
\di X_t = X_{t-} \, \di U_t + \di \eta_t, \quad t\geq 0,
\end{equation*}
cf. \cite[p.428]{MallerMuellerSzimayer2009}.

\medskip
We may now establish Markov processes whose stationary distributions coincide with killed exponential functionals as follows.

\begin{theorem}\label{t-sdewithkilling}
Let $q\in (0,\infty)$, define the L\'evy process $U$ by \eqref{def-U} and let $N$ be a Poisson process with parameter $q$, which is independent of $U$.
 Define
		$$\widetilde{U} := U - N,$$
		then $\widetilde{U}$ is a L\'evy process with L\'evy measure $\nu_{\widetilde{U}}$ such that  $\nu_{\widetilde{U}} (\{-1\}) = q$. Further
\begin{equation}\label{eq-representations}
V_{q,\xi,\eta}\stackrel{d}{=}\int_0^\infty \mathcal{E}(\widetilde{U})_{s-} \, \di \eta_s
\end{equation}
and
$\law(V_{q,\xi,\eta})$ is the unique invariant probability measure of the Markov process ${X=(X_t)_{t\geq 0}}$ satisfying the SDE
	\begin{equation} \label{eq-diff1}
	\di X_t = X_{t-} \, \di \widetilde{U}_t + \di\eta_t, \quad t \geq 0,
	\end{equation}
	with starting random variable $X_0$ independent of $\widetilde{U}$ and $\eta$. The solution of \eqref{eq-diff1} is given~by
\begin{equation} \label{eq-diff1-sol}
X_t = \re^{-\xi_t} X_0 \one_{\{N(t)=0\}} + \re^{-\xi_{t} } \int_{T(t)+}^t \re^{\xi_{s-} } \, \di \eta_s,
\end{equation}
where $T(t)$ denotes the time of the last jump of $N$ before $t$, with the convention that~${T(t)=0}$ if no jump of $N$ occurs before time $t$. In particular,
if $Z$ is a random variable, independent of $(\xi,\eta,N)$ and with the same distribution as $V_{q,\xi,\eta}$, then $Z$ satisfies the random fixed point equation
$$Z \stackrel{d}{=} \re^{-\xi_t} \one_{\{N(t)=0\}} Z + \re^{-\xi_{t} } \int_{T(t)+}^t \re^{\xi_{s-} } \, \di \eta_s$$
for each $t> 0$.
\end{theorem}

\begin{proof}
Observe that $\Delta \widetilde{U}_t =-1$ if and only if $\Delta N_t = 1$, i.e., $N$ counts the jumps of size $-1$ of $\widetilde{U}$, and that	the L\'evy measure $\nu_{\widetilde{U}}$ of $\smash{\widetilde{U}}$ is concentrated on $[-1,\infty)$ with $\nu_{\widetilde{U}}(\{-1\}) = q$. Since $\mathcal{E}(\smash{\widetilde{U}})_t=0$ whenever $N_t\geq 1$ and $\cE(\smash{\widetilde{U}})_t = \cE({U})_t = \re^{-\xi_{t}}$ on $\{N_t=0\}$, and since the time of the first jump of $N$ is exponentially distributed with parameter~$q$,  we obtain Equation \eqref{eq-representations}. By \cite[Thm. 2.2 and Prop. 3.2]{BehmeLindnerMaller2011}, the SDE $\di X_t = X_{t-} \di \widetilde{U}_t + \di \eta_t$ has a strictly stationary solution, unique in distribution, and the corresponding marginal distribution is given by
	$\cE(U)_\tau \smash{\int_0^\tau \cE(U)_{s-}^{-1} \, \di\eta_s},$
where $\tau$ is ${\rm{Exp}}(q)$-distributed, independent from~$(U,\eta)$.
	Further, the process $(X_t)_{t\geq 0}$ defined in \eqref{eq-diff1} is a homogeneous Markov process (cf. \cite[Lem.~3.3]{BehmeLindnerMaller2011}), so that the marginal strictly stationary distribution is the unique invariant probability measure of the Markov process. Since
	$\smash{\cE(U)_t \int_{(0,t]} \cE(U)^{-1}_{s-} \, d\eta_s \stackrel{d}{=} \int_0^t \cE(U)_{s-} \, d\eta_s}$
	for every fixed $t>0$ by \cite[Lem.~3.1]{BehmeLindnerMaller2011}, we obtain that
	$$\cE(U)_\tau \int_0^\tau \cE(U)_{s-}^{-1} \, d\eta_s \stackrel{d}{=} \int_0^\tau \cE(U)_{s-} \, d\eta_s \stackrel{d}{=} V_{q,\xi,\eta}$$
	by conditioning on $\tau$. This shows that $\law(V_{q,\xi,\eta})$ is the unique invariant probability measure of the Markov process $X$.

\medskip
To see the specific form \eqref{eq-diff1-sol} of the solution of \eqref{eq-diff1},
denote for $0\leq s\leq t$
\begin{align*}
\cE(\widetilde{U})_{s,t}&=\exp\Bigl((\widetilde{U}_t-\widetilde{U}_s)-{\sigma_{\widetilde{U}}^2(t-s)}/{2}\Bigr)
\prod_{s<u\leq t}(1+\Delta \widetilde{U}_u)e^{-\Delta \widetilde{U}_u}.
\end{align*}
By \cite[Prop. 3.2, Eq.~(2.7)]{BehmeLindnerMaller2011}, the solution $X=(X_t)_{t\geq 0}$ of \eqref{eq-diff1} is given by
\begin{eqnarray*}
X_t & = & \cE(\widetilde{U})_t \left( X_0 + \int_{0+}^t \left[ \cE (\widetilde{U})_{s-} \right]^{-1} \, \di \eta_s \right) \one_{\{N(t) = 0\}} \\
& & + \cE (\widetilde{U})_{T(t),t} \int_{T(t)+}^t \left[ \cE(\widetilde{U})_{T(t),s-}\right]^{-1} \, \di \eta_s \; \one_{\{N(t)\geq 1\}}.
\end{eqnarray*}
Since $\widetilde{U}_s - \widetilde{U}_{T(t)} = U_s- U_{T(t)}$ for $s\in (T(t),t)$ we see from the Dol\'eans-Dade formula that
$$\cE(\widetilde{U})_{T(t),s} = \re^{-(\xi_s-\xi_{T(t)})} \quad \mbox{for} \quad s\in (T(t), t).$$
Hence, the above can be rewritten (a.s. for fixed $t$) as
$$X_t = \re^{-\xi_t} X_0 \one_{\{N(t)=0\}} + \re^{-(\xi_{t} - \xi_{T(t)})} \int_{T(t)+}^t \re^{\xi_{s-} - \xi_{T(t)}} \, \di \eta_s.$$ Since $\xi$ and $T(t)$ are independent of $\eta$, we can pull out $\re^{-\xi_{T(t)}}$ from the last integral leading to \eqref{eq-diff1-sol}. The desired fixed point equation is now immediate, since $\law(V_{q,\xi,\eta})$ is the invariant probability measure of $X$ with independent starting value $X_0$.
\end{proof}

\begin{remark} \label{rem-killed}
{\rm In the situation of Theorem \ref{t-sdewithkilling}, define the killed L\'evy process $\smash{\widetilde{\xi}}$ with killing rate $q$ and cemetery $\infty$ by
$$\widetilde{\xi}_t = \begin{cases} \xi_t, & t < \tau,\\
+\infty, & t \geq \tau,
\end{cases}$$
where $\tau$ is ${\rm{Exp}}(q)$-distributed, independent of $(\xi,\eta)$. Then
$$V_{q,\xi,\eta} = \int_0^\infty \re^{-\widetilde{\xi}_{s-}} \, \di \eta_s,$$ so that $V_{q,\xi,\eta}$ can be seen as an exponential functional with respect to $\smash{\widetilde{\xi}}$ and $\eta$. The killed L\'evy process $\smash{\widetilde{\xi}}$ and $\smash{\widetilde{U}}$ are related by
$\smash{\cE(\widetilde{U}) = \re^{-\widetilde{\xi}}}$.}
\end{remark}

\begin{remark}
{\rm Since in the situation of Theorem \ref{t-sdewithkilling}, the characteristic triplet of $-N$ is given by $(0,q\delta_{-1},-q)$, where $\delta_{-1}$ denotes the Dirac-measure at $-1$, the characteristic triplet of $\smash{\widetilde{U}}$ can be expressed in terms of the characteristic triplet of $U$ via
$$\sigma_{\widetilde{U}}^2 = \sigma_U^2, \quad \nu_{\widetilde{U}} = \nu_U + q \delta_{-1}, \quad \gamma_{\widetilde{U}} = \gamma_U - q.$$
Additionally, if $U$ is of finite variation, then so is $\widetilde{U}$, and the drifts of $\smash{\widetilde{U}}$ and $U$ are equal. The key difference between the describing SDEs for the exponential functional without killing and the killed exponential functional can then be seen in the L\'evy measure of $U$ and $\smash{\widetilde{U}}$, respectively, since $\nu_U((-\infty,-1])=0$ while $\nu_{\widetilde{U}}(\{-1\}) = q$.}
\end{remark}

 Distributional properties like moments or tail behaviour of the stationary solution of~\eqref{eq-diff1} have been investigated in \cite{Behme2011}. Theorem \ref{t-sdewithkilling} hence allows to apply these results to $V_{q,\xi,\eta}$. Further,  Theorem \ref{t-sdewithkilling} opens the way for deriving various distributional equations for the characteristic function and the distribution function of $V_{q,\xi,\eta}$, in a spirit similar to~\cite{BehmeLindner2015} or \cite{KuznetsovPardoSavov2012}. This will be explored in more detail in the forthcoming paper~\cite{Paper2}.

\section{Supports of killed exponential functionals}\label{S4}

In this section, we consider the support of the distribution of the killed exponential functional~$V:=V_{q,\xi,\eta}$, where we assume throughout that $\xi$ and $\eta$ are independent L\'evy processes with characteristics $(\sigma_\xi^2, \nu_\xi,\gamma_{\xi})$ and $(\sigma_\eta^2, \nu_\eta,\gamma_{\eta})$, respectively, and independent of~$\smash{\tau\overset{d}= \text{Exp}(q)}$ with $q\in (0,\infty)$. As before we denote by $\gamma_\xi^0$, $\gamma_{\eta}^0$ the drift of $\xi$ or $\eta$ whenever it exists.

\medskip
In \cite[Thm. 1]{BehmeLindnerMaejima2016} the support of the exponential functional  $V_{0,\xi,\eta}$ without killing was completely characterised and in particular it was shown that it is always an interval. This is no longer true when considering $q\in (0,\infty)$, as can be seen from the results in this section. We will give a general characterization of the support of $V_{q,\xi,\eta}$ in Theorem~\ref{thm-support} and then study the case where both $\xi$ and $\eta$ are pure compound Poisson processes more closely in Proposition~\ref{thm-support2}. The proofs of both results are postponed to Section \ref{S4proofs}. Note that in case of a non-zero deterministic process~$\eta$, the support of a possibly killed exponential functional has been characterised in \cite[Thm.~2.4(2)]{PatieSavov} in terms of the Wiener-Hopf factorization of~$\xi$, showing that it is always an interval. The special case~$q=0$~(corresponding to $\tau=\infty$) of~\cite[Thm. 2.4(2)]{PatieSavov} was already proven by different means in \cite{BehmeLindnerMaejima2016}. Both results are included in the following theorem. We will, however, present an alternative proof in Section \ref{S4proofs} that does not use the Wiener-Hopf factorization. Observe that Theorem~\ref{thm-support} covers all possible combinations of $\xi$ and $\eta$ and, therefore, provides a complete characterization of the support of the exponential functional in the killed case.

\begin{theorem}[Support of killed exponential functionals]\label{thm-support} $\,$
\begin{itemize}
\item[(i)] If $\eta\equiv0$, then $\supp(V)=\{0\}$.
\item[(ii)] Assume that $\eta$ is deterministic with $\gamma_{\eta}^0>0$, then
\begin{align*}
\supp(V)=\begin{cases}[0,\tfrac{\gamma_{\eta}^0}{\gamma_{\xi}^0}],\ &\text{if}\ \xi\ \text{is\ a\ subordinator\ with}\ \gamma_{\xi}^0>0, \\
[0,\infty),\ &\text{otherwise},\end{cases}
\end{align*}
and if $\gamma_{\eta}^0<0$, then
\begin{align*}
\supp(V)=\begin{cases}[\tfrac{\gamma_{\eta}^0}{\gamma_{\xi}^0},0],\ &\text{if}\ \xi\ \text{is\ a\ subordinator\ with}\ \gamma_{\xi}^0>0, \\
(-\infty,0],\ &\text{otherwise}.\end{cases}
\end{align*}
\item[(iii)] Assume that one of the following cases holds
\begin{itemize}
\item[(a)] $\eta$ is of infinite variation,
\item[(b)] $\eta$ is of finite variation with $0\in\supp(\nu_{\eta})$ and $\nu_{\eta}((-\infty,0))>0$, $\nu_{\eta}((0,\infty))>0$,
\item[(c)] $\eta$ is of finite variation with $\gamma_{\eta}^0\neq0$ and $\nu_{\eta}((-\infty,0))>0$, ${\nu_{\eta}((0,\infty))>0}$,
\end{itemize}
then $\supp(V)=\RR$.
\item[(iv)] Assume that $\eta$ is non-deterministic, of finite variation and spectrally positive/negative, as well as ${0\in\supp(\nu_{\eta})}$ or $\gamma_{\eta}^0\neq0$, then
\begin{align*}
\supp(V)=\begin{cases}[0,\infty),\ &\text{if}\ \eta\ \text{is\ a\ subordinator}, \\
(-\infty,0],\ \ &\text{if}\ -\eta\ \text{is\ a\ subordinator}.\end{cases}
\end{align*}
If under these assumptions neither $\eta$ nor $-\eta$ is a subordinator, we have that
\begin{align*}
\supp(V)=\begin{cases}(-\infty,\tfrac{\gamma_{\eta}^0}{\gamma_{\xi}^0}],\ &\text{if}\ \gamma_{\eta}^0>0\ \text{and}\ \xi\ \text{is\ a\ subordinator\ with}\ \gamma_{\xi}^0>0, \\
[\tfrac{\gamma_{\eta}^0}{\gamma_{\xi}^0},\infty),\ &\text{if}\ \gamma_{\eta}^0<0\ \text{and}\ \xi\ \text{is\ a\ subordinator\ with}\ \gamma_{\xi}^0>0, \\
\RR,\ & otherwise.\end{cases}
\end{align*}
\item[(v)] Assume that $\eta$ is a compound Poisson process with $0\notin\supp(\nu_{\eta})$, then
\begin{equation}\label{support-eta-cpp}
\supp(V)=\overline{\left\{ \sum_{j=1}^n \left(\prod_{k=1}^j a_k \right) b_j,\  a_k>0,\ln a_k \in \Xi, b_j\in \supp(\nu_{\eta}), n\in \NN_0 \right\}},
\end{equation}
where
\begin{equation} \label{def-Xi}
\Xi=\supp(-\xi_T) \quad \text{with} \quad T \overset{d}=  \mathrm{Exp}(\nu_{\eta}(\RR)), \text{ independent of }\xi.
\end{equation}
In particular, if $-\xi$ is a subordinator with~${0\in\supp(\nu_{\xi})}$ or nonzero drift, we have that
\begin{align*}
\supp(V)=\begin{cases}\{0\}\cup [\inf \supp(\nu_{\eta}), \infty), &\text{if}\ \eta\ \text{is\ a\ subordinator},\\
			(-\infty, \sup \supp(\nu_{\eta})]\cup \{0\}, &\text{if}\ -\eta\ \text{is\ a\ subordinator},\\
\RR,\ &\text{otherwise},\end{cases}
\end{align*}
and if $\xi\equiv0$, then $\supp(V) = \overline{ \left\{ \sum_{j=1}^n b_j : b_j \in \supp(\nu_\eta), n\in \bN_0\right\}}$.\\
In the remaining cases, except when $\xi$ is a compound Poisson process with~${0\notin\supp(\nu_{\xi})}$ and only negative jumps,~\eqref{support-eta-cpp} simplifies to
\begin{align*}
\supp(V)=\begin{cases}[0,\infty),\ &\text{if}\ \eta\ \text{is\ a\ subordinator}, \\
(-\infty,0],\ &\text{if}\ -\eta\ \text{is\ a\ subordinator},\\
\RR,\ &\text{otherwise}.\end{cases}
\end{align*}
\end{itemize}
\end{theorem}

After presenting the general result for the support of the killed exponential functional in Theorem~\ref{thm-support}, we will now study the case where both $\xi$ and $\eta$ are compound Poisson processes more closely. Consider the following two motivating examples.
\begin{example}\rm
	Let $M$ and $N$ be two independent Poisson processes. Then $\int_0^t 2^{M_{s-}} \, dN_s$,~${t>0}$, and $\int_0^{\tau} 2^{M_{s-}} \, dN_s$ have support $\NN_0$. More generally, we have for $a>1$ that
\begin{displaymath}
\supp\Big(\int_0^\tau a^{M_{s-}} \, dN_s\Big)=\Big\{ \sum_{k=0}^N  n_k a^k,\ N,n_k \in \NN_0\Big\},
\end{displaymath}
which is neither an interval nor the union of an interval and $\{0\}$.
\end{example}

\begin{example}\rm
Let $\eta$ be a Poisson process and let $\xi$ be a compound Poisson process whose L\'{e}vy measure is supported on the set $-S=-1-C$, where $C$ denotes the classical middle third Cantor set. Thus, both $\supp(\nu_{\eta})=\{1\}$ and $\supp(\nu_{\xi})$ are bounded away from zero and do not contain an interval. Further, $\nu_{\xi}(\RR_+)=0$. By~\cite[Cor.~3.4]{CabrelliHareMolter2002}, we have that~${C+C=[0,2]}$ such that in particular $[2,4]\subseteq\Xi. $ By iteration we further see
that $\Xi=\{0\}\cup S\cup[2,\infty)$ from which we derive by \eqref{support-eta-cpp} that
$$\supp(V)=\overline{\left\{ \sum_{j=1}^n c_j,\, c_j \in \{1\}\cup \re^S \cup [\re^2,\infty), n\in\NN_0  \right\}}$$
as $\eta$ always jumps by $1$. In particular $\supp(V)$ contains the unbounded interval $[\re^2,\infty)$.
	\end{example}

The following proposition collects sufficient conditions for $\supp(V)$ to contain an unbounded interval for $\xi$ and $\eta$ being compound Poisson processes. In particular we will see that if $\supp(\nu_{\xi})$ or $\supp(\nu_{\eta})$ contains an interval, then $\supp(V)$ contains an unbounded interval. Recalling the results of Theorem~\ref{thm-support}, it suffices to consider the case when $0\notin\supp(\nu_{\xi})$ and~$0\notin\supp(\nu_{\eta})$, respectively, as well as $\nu_{\xi}(\RR_+)=0$. Denote by~$\lfloor x\rfloor=\max\{z\in\ZZ: z\leq x\}$ the floor function of $x\in \bR$.

\begin{proposition}\label{thm-support2}
Assume $\xi$ and $\eta$ are independent compound Poisson processes such that $0\not\in\supp(\nu_{\eta})$, $0\not\in \supp(\nu_{\xi})$, and $\nu_{\xi}(\RR_+)=0$. Recall $\supp(\eta_\tau) = \overline{\left\{ \sum_{j=1}^n b_j : b_j \in \supp (\nu_\eta), n\in \bN_0 \right\} }$ from Theorem \ref{thm-support}~(v). 	
			\begin{enumerate}
				\item[(i)] Assume there are $\beta<\alpha<0$ such that $[\beta,\alpha]\subseteq \supp( \nu_\xi)$ and set $k:=\lfloor \frac{\alpha}{\beta-\alpha} \rfloor +1$, then
				 $$\supp(V)\supseteq \begin{cases}
\supp (\eta_\tau)			
\cup [\re^{-k\alpha} \inf \supp(\nu_{\eta}), \infty), & \text{if } \eta \text{ is a subordinator},\\
	\supp(\eta_\tau)			 	\cup (-\infty, \re^{-k\alpha} \sup \supp(\nu_{\eta})], & \text{if } -\eta \text{ is a subordinator},\\
				 \RR & \text{otherwise.} \end{cases}$$
				\item[(ii)] Assume $\eta$ is a subordinator and there are $0<\alpha<\beta$ such that $[\alpha,\beta]\subseteq\supp(\nu_{\eta})$. Set~$k:=\lfloor \frac{\alpha}{\beta-\alpha} \rfloor +1$, then
				$$\supp(V)\supseteq \begin{cases}
				\{0\}\cup [\alpha,\infty) ,& \text{if }  \ln (\frac{\beta}{\alpha}) \geq - \sup \supp(\nu_\xi), \\
				\{0\}\cup \bigcup_{\ell=1}^{k-1}[\ell \alpha,\ell \beta]\cup [k\alpha,\infty), & \text{otherwise.}
				\end{cases}$$
				\item[(iii)] Assume $-\eta$ is a subordinator and there are $\beta<\alpha<0$ such that $[\beta,\alpha]\subseteq\supp(\nu_{\eta})$. Set $k:=\lfloor \frac{\alpha}{\beta-\alpha} \rfloor +1$, then
				$$\supp(V)\supseteq \begin{cases}
				(-\infty,\alpha]\cup \{0\} ,& \text{if }  \ln (\frac{\beta}{\alpha}) \geq - \sup \supp(\nu_\xi), \\
				(-\infty,k\alpha] \cup \bigcup_{\ell=1}^{k-1}[\ell \beta,\ell \alpha]\cup   \{0\}, & \text{otherwise.}
				\end{cases}$$
				\item[(iv)] Assume $\nu_{\eta}(\RR_-)\neq 0 \neq \nu_{\eta}(\RR_+)$, and there are $0<\alpha<\beta$ such that ${[\alpha,\beta]\subseteq\supp(\nu_{\eta})}$ or $\beta<\alpha<0$ such that $[\beta,\alpha]\subseteq\supp(\nu_{\eta})$. Then
				$$\supp(V)=\RR.$$
\item[(v)] Assume $\nu_{\eta}(\RR_-)\neq 0 \neq \nu_{\eta}(\RR_+)$ and that there are numbers $z_1<0$, $z_2>0$ in $\supp(\nu_{\eta})$ such that~$\frac{z_2}{z_1}$ is irrational. Then
				$$\supp(V)=\RR.$$
\end{enumerate}
\end{proposition}

\begin{example}\rm
Let $\xi$ be a Poisson process and let $\eta$ be a compound Poisson process with L\'{e}vy measure $\nu_{\eta}=\delta_{-1}+\delta_{\sqrt{2}}$, where $\delta_x$ denotes the Dirac measure supported at $x\in\RR$. Thus, both $\supp(\nu_{\xi})=\{1\}$ and $\supp(\nu_{\eta})=\{-1,\sqrt{2}\}$ are bounded away from zero and do not contain an interval. However, $\smash{\frac{\sqrt{2}}{-1}}\in\RR\setminus\QQ$ and thus $\supp(V)=\RR$ by Part~(v) of Proposition~\ref{thm-support2}.
\end{example}


\section{Continuity of (killed) exponential functionals} \label{S5neu}

In this section we present conditions for (absolute) continuity of killed exponential functionals $V_{q,\xi,\eta}$ as well as of exponential functionals $V_{0,\xi,\eta}$ without killing. We start with a collection of sufficient conditions for absolute continuity of killed exponential functionals given in Theorem \ref{t-ac3}. Exponential functionals without killing will be considered likewise in Theorem \ref{t-ac5}. The proofs of the results of this section will be given in Section \ref{S5newproofs}.

\medskip
As before, throughout this section, $\xi$ and $\eta$ denote independent L\'evy processes with characteristic triplets $(\sigma_\xi^2, \nu_\xi,\gamma_\xi)$ and  $(\sigma_\eta^2, \nu_\eta,\gamma_\eta)$, respectively, and characteristic exponents $\Psi_\xi$ and $\Psi_\eta$, respectively, independent of $\tau \overset{d}= \text{Exp}(q)$ for $q\in (0,\infty)$.
The Lebesgue decompositions of $\nu_\xi$ and $\nu_\eta$ will be denoted by
	$$\nu_\xi = \nu_{\xi,{\rm{ac}}} + \nu_{\xi,{\rm{sing}}} \quad \mbox{and} \quad
	\nu_\eta = \nu_{\eta,{\rm{ac}}} + \nu_{\eta,{\rm{sing}}},$$
respectively, where \lq\lq ac\rq\rq~marks the absolutely continuous and \lq\lq sing\rq\rq~the singular part. Further, $\Re(z)$ denotes the real part of a complex number $z$. Lastly, for each $p\in [0,\infty)$, the \emph{$p$-potential measure $W_\eta^p$ of $\eta$} is defined by
\begin{equation} \label{eq-potential}
		W_\eta^p (B) = \int_0^\infty \re^{-p t} P(\eta_t \in B) \, \di t= E \left( \int_0^\infty \re^{-pt} \one_B(\eta_t) \, \di t \right), \quad B \in \cB_1,
\end{equation}
e.g., \cite[Def. 30.9]{Sato2013}; in other literature such as \cite{Bertoin1996} this appears also under the name of \emph{resolvent kernel} when $p>0$. Likewise, the potential measure of $\xi$ will be denoted as $W_\xi^p$.

\subsection{(Absolute) continuity of the killed exponential functional}

\begin{theorem}[Sufficient conditions for absolute continuity of $V_{q,\xi,\eta}$]\label{t-ac3}
 Suppose that one of the following conditions  is satisfied:
	\begin{enumerate}
		\item[(i)] The characteristic triplet of $\eta$ satisfies Kallenberg's condition (\cite[pp. 794--795]{Kallenberg1981})
		\begin{equation} \label{eq-Kallenberg5}
				\lim_{\varepsilon \downarrow 0} \varepsilon^{-2} |\ln \varepsilon|^{-1}\left( \sigma^2 + \int_{-\varepsilon}^{\varepsilon} x^2 \, \nu (\di x)\right) = \infty,
		\end{equation}
	 or more generally the Hartman--Wintner condition (\cite[pp. 287--288]{HartmanWintner1942})
	 \begin{equation} \label{eq-HW3}
	 		\lim_{|z|\to\infty} \frac{-\Re (\Psi(z))}{\ln (1+|z|)}=\infty.
	 \end{equation}
		In particular, this is satisfied when $\sigma_\eta^2 > 0$.
		\item[(ii)] The absolutely continuous part of $\nu_\eta$ is infinite: $\nu_{\eta,{\rm{ac}}} (\bR) = \infty$.
		\item[(iii)] The characteristic exponent $\Psi_\xi$ of $\xi$ satisfies
		\begin{equation} \label{eq-Hawkes}
				\int_\bR \Re \left( \frac{1}{1-\Psi_\xi (z)} \right) \, \di z < \infty,
		\end{equation}
	and ${\nu_\eta(\bR) = \infty}$.
		\item[(iv)] $\nu_{\xi,{\rm ac}}(\bR) = \nu_\eta(\bR) = \infty$.
		\item[(v)] $\eta$ is of finite variation with  non-zero drift.
		\item[(vi)] $\xi$ is a compound Poisson process and $\eta$ satisfies the ACP condition:
		\begin{equation} \label{eq-ACP}
				W_\eta^p\text{ is absolutely continuous for some } p\in [0,\infty).
		\end{equation}
	\end{enumerate}
	Then $V_{q,\xi,\eta} = \int_0^\tau \re^{-\xi_{s-}} \, d\eta_s$ is absolutely continuous.
\end{theorem}

A discussion of the various assumptions of this theorem, in particular examples for the validity of \eqref{eq-Hawkes} and \eqref{eq-ACP}, will be given in Section \ref{S-discussion}.
When $\eta$ is deterministic but not the zero-process, Pardo et al. \cite[Thm. 2.1]{PardoRiveroSchaik} showed that $V_{q,\xi,\eta}$ has a density and they also obtained various properties of it. Observe that the existence of the density in this case can also be seen from Theorem \ref{t-ac3}~(v).

\begin{remark}\label{rem-ext-new} \rm
	Many of the results of Theorem \ref{t-ac3} can be extended to functionals of the form $\int_0^\tau g(\xi_{s-}) \, \di \eta_s$ for sufficiently nice functions $g$. For example, if $g:[0,\infty) \to \bR$ is continuous, $\xi$ satisfies
	\begin{equation} \label{eq-g-N}
			\int_0^t P(\xi_{s} \in N) \, \di s = 0.
	\end{equation}
		 for the zero set $N$ of $g$, and one of the conditions (i) -- (iii) of Theorem \ref{t-ac3} is satisfied, then also $\int_0^\tau g(\xi_{s-}) \, \, \di \eta_s$ will be absolutely continuous by the same proof.\\	
		 When $N$ is countable,  sufficient conditions for \eqref{eq-g-N} are that $\sigma_\xi^2>0$ or $\nu_\xi(\bR) = \infty$ (by \cite[Thm. 27.4]{Sato2013}), or that $\xi$ is of finite variation with non-zero drift since then \eqref{eq-Hawkes} is satisfied which gives absolute continuity of the potential measure $W^1_\xi$ and hence of $W^0_\xi$. Another sufficient condition obviously is that $g\neq 0$ on $[0,t]$.
\end{remark}

Moreover we highlight the following two corollaries of  Theorem \ref{t-ac3}.

\begin{corollary} \label{c-ac4}
	Assume that $\sigma_\xi^2 > 0$.
	Then $V_{q,\xi,\eta}$ is absolutely continuous if and only if $\eta$ is neither a compound Poisson process nor the zero process.
\end{corollary}

When $\xi$ is a compound Poisson process, it is easy to show that the sufficient condition~(vi) of Theorem \ref{t-ac3} is actually also necessary:

\begin{corollary} \label{c-ac7}
	Let $\xi$ be a compound Poisson process.
	 Then $V_{q,\xi,\eta}$ is absolutely continuous if and only if $\eta$ satisfies the ACP condition \eqref{eq-ACP}.
\end{corollary}

Continuity of $V_{q,\xi,\eta}$ is characterised in the following proposition.

\begin{proposition}[Continuity of $V_{q,\xi,\eta}$] \label{c-ac5}
$V_{q,\xi,\eta}$ is continuous if and only if $\eta$ is neither a compound Poisson process nor the zero process. If $\eta$ is a compound Poisson process or the zero process, then $V_{q,\xi,\eta}$ has an atom at zero.
\end{proposition}

A probability law on $\bR$ is said to be of \emph{pure type} if it is either discrete, continuous singular or absolutely continuous. The next example shows that $V_{q,\xi,\eta}$ is not always of pure type (unlike $V_{0,\xi,\eta}$, see Section \ref{S-5.2}).

\begin{example}
	{\rm If $\eta$ is a compound Poisson process, then $V_{q,\xi,\eta}$ has an atom at zero and hence trivially cannot be absolutely continuous. But its distribution restricted to $\bR^*$ can be absolutely continuous, as we show now.
		Denote the time of the first jump of $\eta$ by $R$. Conditional on $\{\tau > R\}$, we can write
		\begin{equation*}
				V_{q,\xi,\eta}  =  \re^{-\xi_R} \, \Delta \eta_R + \int_{R+}^\tau \re^{-\xi_{s-}} \, \di \eta_s
				= \re^{-\xi_R} \left( \Delta \eta_R + \int_{R+}^{R + (\tau - R)} \re^{-(\xi_{s-} - \xi_R)} \, \di \eta_s\right),
		\end{equation*}
		with $\re^{-\xi_R}$, $\Delta \eta_R$ and $\int_{R+}^{R+(\tau - R)} \re^{-(\xi_{s-} - \xi_R)} \, \di \eta_s$ being conditionally independent by the strong Markov property of L\'evy processes. If now the jump distribution of $\eta$ is absolutely continuous, we conclude that $V_{q,\xi,\eta}$ conditional on $\{\tau > R\}$ is absolutely continuous. The same is true if $\xi$ satisfies the ACP-condition, since then $\xi_R$ is absolutely continuous (see the discussion after Example \ref{ex-Hawkes} below) and hence so is $\re^{-\xi_R}$. In particular, the law of $V_{q,\xi,\eta}$ is not of pure type.
	}
\end{example}

We end our study of killed functionals with an example where $V_{q,\xi,\eta}$ is continuous but not absolutely continuous.

\begin{example} \label{ex-no-density1}
	{\rm Let $0<\alpha < 1$, $c$ be an integer such that $c> 1/(1-\alpha)$, let $a_n = 2^{-c^n}$ for $n\in \bN$ and define the L\'evy measure $\nu_\eta$ by $\nu_\eta := \sum_{n=1}^\infty a_n^{-\alpha} \delta_{a_n}$. Then  $\nu_\eta$ is infinite with $\int_{|x|\leq 1} |x|\, \nu_\eta(\dx x) < \infty$. Let $\eta$ be the subordinator with L\'evy measure $\nu_\eta$ and drift 0. According to the final part of Example 41.23 in Sato \cite{Sato2013}, the potential measure $W^q_\eta$ of $\eta$ is continuous singular for any $q>0$.
		Now let $\tau$ be an exponentially distributed random variable with parameter $q>0$ and let $\xi$ be the zero process. Then $V_{q,\xi,\eta} = \eta_\tau$ which has the same distribution as $q W_\eta^q$ (see the discussion of the ACP-condition after Example \ref{ex-Hawkes} below). It follows that $V_{q,\xi,\eta}$ is continuous singular.

\medskip
Further, to obtain an example with non-deterministic integrand, let $\xi'$ be a compound Poisson process and denote by $T$ the time of its first jump. With positive probability,~$\xi'$ does not jump before time $\tau$, i.e., $\{T > \tau\}$ has positive probability, and on this set we have $V_{q,\xi',\eta} = \eta_\tau$. Since conditionally on $\{T> \tau\}$, $\tau$ is exponentially distributed with parameter $q+\nu_\xi(\bR)$, also the conditional distribution of $\eta_\tau$ given $\{T > \tau\}$ is continuous singular. We conclude that $V_{q,\xi',\eta}$ has a non-trivial singular part.}
\end{example}

\subsection{Absolute continuity of the exponential functional without killing} \label{S-5.2}

We now turn our attention to the exponential functional without killing and assume additionally from now on that
$V_{0,\xi,\eta} := \int_0^\infty \re^{-\xi_{s-}} \, \di \eta_s$
converges a.s. and $\eta$ is not the zero process. A characterisation when the integral converges in terms of the characteristic triplet of $\xi$ and $\eta$ is given by Erickson and Maller \cite[Thm. 2]{EricksonMaller2005}. In particular, $\xi$ has to drift a.s. to $\infty$, which implies that it is transient.

\medskip
Although much more attention has been paid to $V_{0,\xi,\eta}$ rather than $V_{q,\xi,\eta}$ when $q>0$, not too many sufficient conditions for absolute continuity of $V_{0,\xi,\eta}$ are known. Bertoin et al. \cite[Thm. 3.9 (a)]{BertoinLindnerMaller2008} show that if $\eta$ is of finite variation with non-zero drift and $\nu_\xi(\bR) > 0$, then $V_{0,\xi,\eta}$ will be absolutely continuous. They also characterise continuity of $V_{0,\xi,\eta}$ and show that it is always continuous unless both $\xi$ and $\eta$ are deterministic, cf. \cite[Thm. 2.2]{BertoinLindnerMaller2008}. Kuznetsov et al. \cite[Cor. 2.5]{KuznetsovPardoSavov2012} find that $V_{0,\xi,\eta}$ has a density whenever $\sigma_\eta^2 + \sigma_\xi^2 > 0$ and $\eta$ and $\xi$  both have finite expectation. Also, it is known that when $\xi$ is spectrally negative, then $V_{0,\xi,\eta}$ is self-decomposable (\cite[Rem. (i) in Sect. 2]{BertoinLindnerMaller2008}), and hence absolutely continuous unless it is constant (e.g., \cite[Ex. 27.8]{Sato2013}), i.e., unless both $\xi$ and $\eta$ are deterministic.

\medskip
It is also known that the law of $V_{0,\xi,\eta}$ is of pure type, and even that it is either degenerate, or continuous singular, or absolutely continuous, e.g., \cite[Sect. 5]{BehmeLindnerMaller2011}. In \cite{LindnerSato2009} the law of $\int_0^\infty \re^{- (\ln c)N_{t-}} \, \di \eta_t$ is studied when $\eta$ and $N$ are two independent Poisson process and $c>1$. It is shown that the distribution in that case may be continuous singular or absolutely continuous, depending in an intrinsic way on algebraic properties of $c$ and the ratio of the rates of the two Poisson processes $N$ and $\eta$ (cf. \cite[Thms. 3.1, 3.2]{LindnerSato2009}). The question whether $V_{0,\xi,\eta}$ will always be absolutely continuous for general L\'evy processes $\xi$ and $\eta$ that have both infinite L\'evy measure (or only one of them) is still open. Still, in the following theorem, we collect various sufficient conditions for absolute continuity of $V_{0,\xi,\eta}$, many of which are new.

\begin{theorem}[Sufficient conditions for absolute continuity of $V_{0,\xi,\eta}$]\label{t-ac5}
Suppose that one of the following conditions  is satisfied:
	\begin{enumerate}
		\item[(i)] The characteristic triplet of $\eta$ satisfies
		\begin{equation} \label{eq-Kallenberg11}				\liminf_{\varepsilon \downarrow 0} \varepsilon^{-2} |\ln \varepsilon|^{-1}\left( \sigma^2_\eta + \int_{-\varepsilon}^\varepsilon x^2 \, \nu_\eta(d x)\right) > 0,
		\end{equation}
		or more generally
		\begin{equation} \label{eq-HW11}
				\liminf_{|z|\to \infty} \frac{- \Re (\Psi_\eta(z))}{\ln (1+|z|)} > 0.
		\end{equation}
		In particular, the conditions are satisfied when $\sigma_\eta^2 > 0$.
		\item[(ii)] The absolutely continuous part of $\nu_\eta$ is non-trivial: $\nu_{\eta,{\rm{ac}}} (\bR) > 0$.
		\item[(iii)] The characteristic exponent $\Psi_\xi$ of $\xi$ satisfies Condition \eqref{eq-Hawkes}, and at least one of $\xi$ and $\eta$ is non-deterministic.
		\item[(iv)] The absolutely continuous part of $\nu_\xi$ is non-trivial: $\nu_{\xi,{\rm ac}}(\bR) > 0$.
		\item[(v)] $\eta$ is of finite variation with  non-zero drift, and at least one of $\xi$ and $\eta$ is non-deterministic.
		\item[(vi)] $\xi$ is a compound Poisson process and $\eta$ satisfies the ACP condition \eqref{eq-ACP}.
		\item[(vii)] $\eta$ is a compound Poisson process and $\xi$ satisfies the ACP condition \eqref{eq-ACP}.
		\item[(viii)] $\xi$ is spectrally negative, and at least one of $\xi$ and $\eta$ is non-deterministic.
	\end{enumerate}
	Then $V_{0,\xi,\eta} = \int_0^\infty \re^{-\xi_{s-}} \, d\eta_s$ is absolutely continuous.
\end{theorem}

Observe that part (viii) above is already covered by parts (i), (iii) and (v), for if $\xi$ is spectrally negative and drifts to infinity, then it is either of finite variation with strictly positive drift, or it is of infinite variation with $\nu_\xi((0,\infty)) = 0$, so that in both cases $\Psi_\xi$ satisfies Condition \eqref{eq-Hawkes} and  (viii) follows from (iii) when $\nu_\eta(\bR) > 0$. When $\nu_\eta(\bR) = 0$, (viii) follows from (i) and (v).

\medskip
The following result generalises Corollary 2.5 of Kuznetsov et al. \cite{KuznetsovPardoSavov2012} in the sense that it shows that the assumption in \cite{KuznetsovPardoSavov2012} that both $\xi$ and $\eta$ have finite expectation can be omitted for the existence of a density.

\begin{corollary} \label{c-ac6}
Suppose that $\sigma_\eta^2+\sigma_\xi^2 > 0$. Then $V_{0,\xi,\eta}$ is absolutely continuous.
\end{corollary}

A result similar to Corollary \ref{c-ac7} does not hold when $q=0$. This follows by observing that $\int_0^\infty \re^{-\ln(c) N_{t-}} \, \di \eta_t$ can be absolutely continuous for suitable constants $c>1$ and Poisson processes $N$ and $\eta$ by \cite[Thm. 3.2]{LindnerSato2009}; obviously, a Poisson process does not satisfy the ACP condition.

\medskip
Let us finally mention that similar to Remark \ref{rem-ext-new}, some of the results of Theorem \ref{t-ac5} can be easily extended to functionals of the form $\smash{\int_0^\infty g(\xi_{s-}) \, \di \eta_s}$, assuming the convergence of the integral. In particular, when $g:[0,\infty) \to \bR$ is continuous with zero set $N$, if the integral converges, if $\smash{\int_0^\infty P(\xi_s \in N) \, \di s} = 0$ (i.e., \eqref{eq-g-N}) and if one of the conditions (i)--(iii) of Theorem \ref{t-ac5} is satisfied, then $\smash{\int_0^\infty g(\xi_{s-}) \, \di \eta_s}$ will be absolutely continuous. Sufficient conditions for \eqref{eq-g-N} have been discussed in Remark \ref{rem-ext-new}.

\subsection{Discussion and further sufficient conditions for (absolute) continuity} \label{S-discussion}

We end this section with some remarks concerning the various assumptions appearing in Theorems \ref{t-ac3} and \ref{t-ac5} above.

\medskip
Kallenberg's condition \eqref{eq-Kallenberg5} is a classical condition for absolute continuity of an infinitely divisible distribution. His proof (\cite[pp.794--795]{Kallenberg1981}) shows that it implies the Hartman-Wintner condition \eqref{eq-HW3}.

\medskip
To interpret Condition \eqref{eq-Hawkes}, for each path $f$ of $\xi$, we define the occupation measure~$\rho_{f,t}$ on $\cB_1$ by
	$$\rho_{f,t}(B) = \int_0^t \one_B (f(s)) \, \di s, \quad B \in \cB_1.$$
Then, as shown by  Hawkes \cite{Hawkes1985} (cf. Bertoin \cite[Thm. V.1]{Bertoin1996}), Condition \eqref{eq-Hawkes} is equivalent to the fact that for $P_\xi$-almost every path $f$ of $\xi$, the occupation measure $\rho_{f,t}$ is absolutely continuous. This in turn is equivalent to saying that for $P_\xi$-almost every path $f$ of~$\xi$, preimages of Lebesgue nullsets under the mapping $f:[0,t] \to \bR$ are again Lebesgue nullsets, i.e., that $f:[0,t]\to \bR$ satisfies the so-called Lusin$(N^{-1})$-condition. A further equivalent condition can be expressed in terms of potential measures. Condition \eqref{eq-Hawkes} is then equivalent to the fact that $W^p_\xi$ has a bounded Lebesgue density for some, equivalently all, $p>0$, cf.  \cite[Thm. II.16]{Bertoin1996} or \cite[Thm. 43.3, Rem. 43.6]{Sato2013}. Finally, Condition~\eqref{eq-Hawkes} is further equivalent to the fact that
single points are not essentially polar under $\xi$, equivalently that the $p$-capacity $C^p(\{0\})$ of $\xi$ is strictly positive for some, equivalently all, $p>0$, see \cite[Sect. II.3]{Bertoin1996} or \cite[Def. 41.14, 42.6]{Sato2013} for the definitions of essentially polar sets and the $p$-capacity and \cite[Thm. II.16]{Bertoin1996} or \cite[Prop. 43.2, Thm. 43.3]{Sato2013} for the corresponding results. We collect some known examples when Condition \eqref{eq-Hawkes} is satisfied:

\begin{example}[Sufficient conditions for \eqref{eq-Hawkes}] \label{ex-Hawkes}
	{\rm Let $\xi$ be a L\'evy process. If $\xi$ is of finite variation, then \eqref{eq-Hawkes} holds if and only if the drift $\gamma_\xi^0$ of $\xi$ is different from zero (\cite[Cor.~II.20~(ii)]{Bertoin1996}, \cite[Thm. 43.13]{Sato2013}). Condition \eqref{eq-Hawkes} also holds if $\sigma_\xi^2>0$ (\cite[Thm. 43.21 Case~6]{Sato2013}) or more generally if $\xi$ is $\alpha$-stable with index $\alpha \in (1,2]$ (\cite[Ex. 43.22]{Sato2013}). A non-deterministic 1-stable process $\xi$ satisfies \eqref{eq-Hawkes} if and only if it is not strictly 1-stable, cf.~\cite[Ex. 43.7]{Sato2013}. Condition \eqref{eq-Hawkes} is further satisfied, when
		\begin{equation} \label{eq-Hawkes2}
				\int_0^1 x \, \nu_\xi(\di x) < \infty = \int_{-1}^0 |x| \, \nu_\xi(\di x) \quad \mbox{or} \quad
				\int_{-1}^0 |x| \, \nu_\xi(\di x) < \infty = \int_0^1 x \nu_\xi(\di x),
		\end{equation}
		cf. \cite[Thm. 43.24]{Sato2013}.}
\end{example}

Finally, the \emph{ACP condition} \eqref{eq-ACP} is equivalent to saying that $W^p_\eta$ is absolutely continuous for \emph{all} $p\in [0,\infty)$, cf. \cite[Rem.~41.12]{Sato2013}. If $p\in (0,\infty)$ and $T$ is an exponentially distributed time with parameter~$p$, independent of $\eta$, then by conditioning on $T=t$ we have for all $B\in \cB_1$

	$$P( \eta_T \in B ) = \int_0^\infty P( \eta_t \in B) \, P_T(\di t) = p \int_0^\infty \re^{-pt} P(\eta_t \in B)\, \di t = p W^p_\eta(B).$$
Hence \eqref{eq-ACP} means nothing else than that $\eta_T$ is absolutely continuous for any (equivalently: some) exponentially distributed independent time $T$ with parameter in $(0,\infty)$. Obviously,~\eqref{eq-ACP} is satisfied when $\eta_t$ is absolutely continuous for each $t>0$, but the converse is not true in general, see, e.g., \cite[Rem.~41.13]{Sato2013}.
Recall also from above that $\int_\bR \Re (\frac{1}{1-\Psi_\eta(z)}) \, \di z < \infty$ if and only if $W^1_\eta$ has a \emph{bounded} density, so that \eqref{eq-Hawkes} for $\eta$ implies~\eqref{eq-ACP} for $\eta$.
In particular, $\eta$ satisfies \eqref{eq-ACP} if $\eta$ is of finite variation with non-zero drift, if $\sigma_\eta^2>0$, or if $\eta$ satisfies \eqref{eq-Hawkes2} (with $\nu_\eta$ replacing $\nu_\xi$). Examples exist when $\eta$ satisfies \eqref{eq-ACP} but not \eqref{eq-Hawkes}, see \cite[Thm. 43.21]{Sato2013}; one such example is when $\eta$ is a non-deterministic strictly $\alpha$-stable process of index $\alpha \in (0,1)$.

\section{Conditions for (absolute) continuity of other related integrals} \label{S6new}
In this section we will consider integrals as they appear when a (killed) exponential function is treated either conditionally on the paths of $\xi$, or on the killing time. More precisely, the integral $\smash{\int_0^t \re^{-\xi_{s-}} \, \di \eta_s}$ conditional on a path $\xi=f$ is equal to the integral $\smash{\int_0^t \re^{-f(s-)} \, \di \eta_s}$, and hence the improper integral $V_{0,\xi,\eta} = \smash{\int_0^\infty \re^{-\xi_{s-}} \, \di \eta_s}$ conditional on $\xi=f$ is equal to the improper integral $\smash{\int_0^\infty \re^{-f(s-)}\, \di \eta_s}$, and similarly for killed exponential functionals. It will then be possible in the proofs given in Section \ref{S5newproofs} to deduce some of the continuity properties stated in Section \ref{S5neu} of (killed) exponential functionals from those of the corresponding conditioned integrals.


\subsection{Conditions for absolute continuity of  $\int_0^\infty f(s) \, \di \eta_s$} \label{S5}

Throughout, let $\eta=(\eta_t)_{t\geq 0}$ be a one-dimensional L\'evy process with characteristic triplet $(\sigma_\eta^2,\nu_\eta, \gamma_\eta)$and ${f:[0,\infty) \to \bR}$ a deterministic Borel measurable function. We say that $f$ is \emph{locally integrable}, or more precisely, \emph{locally integrable with respect to the independently scattered random measure induced by $\eta$},  if $f\one_{[0,t]}$ is integrable with respect to $\eta$ in the sense of Raj\-put and Rosinski (cf. \cite[p.~460]{RajputRosinski1989}, or also Sato \cite[Def. 57.8]{Sato2013}) for every $t\in (0,\infty)$. The corresponding integral over $[0,t]$ is denoted by $\int_0^t f(s )\, \di \eta_s$. Since $t\mapsto \int_0^t f(s) \, \di \eta_s$ defines an additive process, a version of the integral exists that has c\`adl\`ag paths (e.g.,~\cite[Thm.~11.5]{Sato2013}), and we shall always assume that such a version is chosen. We say that the \emph{improper integral $\int_0^{\infty} f(s) \, \di \eta_s$ exists}, if $f$ is locally integrable and $\int_0^t f(s) \, \di \eta_s$ converges in probability to a finite random variable as $t\to\infty$, equivalently by the independent increments property and the c\`adl\`ag paths, if it converges a.s. to a finite random variable as $t\to\infty$; we denote the limit by $\int_0^\infty f(s) \, \di\eta_s$.  A characterisation of functions for which the improper integral $\int_0^\infty f(s) \, \di \eta_s$ exists can be found in Sato \cite[Prop. 5.5]{Sato2006} or~\cite[Prop.~57.13]{Sato2013}. In particular,
every locally bounded  measurable function is locally integrable, and for functions with bounded support, integrability as defined in Rajput and Rosinski~\cite[p.~460]{RajputRosinski1989} is equivalent to the existence of the improper integral.

\medskip
When the improper integral $\int_0^{\infty} f(s) \, \di\eta_s$ exists, its distribution is infinitely divisible with characteristic triplet $(\sigma_f^2, \nu_f, \gamma_f)$, where $\sigma_f^2 = \sigma_\eta^2 \int_0^\infty f(s)^2 \, \di s$ and $\nu_f$ is given by \eqref{eq-LM1} with $t=\infty$, i.e.,
\begin{equation} \label{eq-LM1a}
\nu_f(B)  =  \int_0^\infty \int_\bR \one_{B\setminus \{0\}} (f(s) x) \, \nu_\eta(\di x) \, \di s,  \quad B \in \cB_1.
\end{equation}
Further, the characteristic exponent $\Psi_f$ of $\int_0^{\infty} f(s)\, \di \eta_s$ is given by
\begin{equation}
\Psi_f (z) = \lim_{t\to\infty} \int_0^t \Psi_\eta(f(s) z) \, \di s, \label{eq-LM2}
\end{equation}
cf. \cite[Prop. 57.13]{Sato2013}. We can hence apply continuity results for infinitely divisible distributions.
We start with a simple result characterising continuity of $\int_0^\infty f(s) \, \di \eta_s$, i.e., when the distribution has no atoms.

\begin{proposition}[Continuity of $\int f(s) \, \di \eta_s$] \label{p-cont1}Let $f:[0,\infty) \to \bR$ be a deterministic Borel measurable function such that the improper integral $\int_0^\infty f(s) \, \di \eta_s$ exists and such that $f\neq 0$ on a set of positive Lebesgue measure. Then $\int_0^\infty f(s) \, \di \eta_s$ is continuous (i.e., has no atoms) if and only if
$$\sigma_\eta^2 > 0, \quad \mbox{or} \quad
 \lambda^1 (\{ s \in [0,\infty) : f(s) \neq 0\}) \cdot \nu_\eta(\bR)= \infty.$$
\end{proposition}

\begin{proof}
By \cite[Thm. 27.4]{Sato2013}, $\int_0^\infty f(s) \, \di \eta_s$ is continuous if and only if $\sigma_f^2>0$ or $\nu_f(\bR) = \infty$. Since $f$ is not Lebesgue almost everywhere equal to zero, the claim follows by observing that  $\sigma_f^2 = \sigma_\eta^2 \int_0^\infty f(s)^2 \, \di s$ and
$\nu_f(\bR) = \lambda^1 (\{ s\in [0,\infty): f(s) \neq 0\}) \, \nu_\eta(\bR)$ by \eqref{eq-LM1a}.
\end{proof}

In the following proposition we collect some known results ensuring absolute continuity of infinitely divisible distributions. As before, we denote by $\Re(z)$ the real part of a complex number $z$.

\begin{lemma}[Absolute continuity of infinitely divisible distributions] \label{prop-id-ac}
Let $\mu$ be an infinitely divisible distribution with characteristic triplet $(\sigma^2,\nu,\gamma)$ and characteristic exponent $\Psi$.
\begin{itemize}
\item[(i)] If Kallenberg's condition \eqref{eq-Kallenberg5} or more generally the Hartman-Wintner condition  \eqref{eq-HW3} is satisfied, then $\mu$ is absolutely continuous with infinitely often differentiable density with all derivatives vanishing at infinity.
\item[(ii)] If
\begin{equation} \label{eq-David1}
\lim_{\varepsilon \downarrow 0} \varepsilon^{-2} |\ln \varepsilon|^{-1}\left( \sigma^2 + \int_{-\varepsilon}^{\varepsilon} x^2 \, \nu (\di x)\right) > 1/4,
\end{equation}
or more generally
\begin{equation} \label{eq-David2}
\liminf_{|z|\to\infty} \frac{-\Re (\Psi(z))}{\ln (1+|z|)}> 1/2,
\end{equation}
then $\mu$ is absolutely continuous with square integrable density (\cite[Cor. 3.6]{Berger2018}).
\item[(iii)] If the absolutely continuous part $\nu_{\rm{ac}}$ in the Lebesgue decomposition of $\nu$ is infinite, then $\mu$ is absolutely continuous (\cite[Thm. 27.7]{Sato2013}).
\end{itemize}
\end{lemma}

As mentioned, Kallenberg proof (\cite[pp.~794--795]{Kallenberg1981}) shows that \eqref{eq-Kallenberg5} implies the Hartman-Wintner condition \eqref{eq-HW3} which in turn implies that the Fourier transform $\widehat{\mu}$ of $\mu$ satisfies $\int_\bR |x|^k \, |\widehat{\mu}(x)| \, \di x < \infty$ for all $k\in \bN$, giving (i). It has been noted by several authors that if the right-hand sides of \eqref{eq-David1} and \eqref{eq-David2} are replaced by $>1$ (e.g.,  \cite[p.~853]{Chang2017}, \cite[p.~127]{KnopovaSchilling2013} for \eqref{eq-David1} and Hartman and Wintner \cite[pp. 794--795]{HartmanWintner1942} themselves for~\eqref{eq-David2}), then $\mu$ has a continuous and bounded density vanishing at infinity. The fact that the constant 1 can even be replaced by $1/4$  and $1/2$ (as done in \eqref{eq-David1} and \eqref{eq-David2} in (ii)) to ensure absolute continuity has been shown by Berger \cite[Cor. 3.6]{Berger2018}. In fact, Berger's proof shows that \eqref{eq-David1} implies \eqref{eq-David2}, which in turn implies square integrability of the Fourier transform of $\mu$, thus giving absolute continuity of $\mu$ with square integrable density (cf.~\cite[Thm. 11.6.1]{Kawata1972}). Part (iii) is an easy consequence of Sato \cite[Thm. 27.7]{Sato2013}, by observing that the convolution of an absolutely continuous distribution with another distribution is again absolutely continuous.

\medskip
An interesting feature of integrals of the form $\int_0^\infty f(s) \, \di \eta_s$ is that their distribution is often smoother than the original distribution. A well known example is $\int_0^\infty \re^{-as} \, \di \eta_s$ whenever $a>0$ and $\eta$ is a non-deterministic L\'evy process such that the integral converges, which always gives a self-decomposable and hence absolutely continuous distribution. That this phenomenon can happen also with functions with compact support was exemplified by Nourdin and Simon \cite[Thm. A]{NourdinSimon2006}, who showed that $\int_0^t \re^{-as} \, \di \eta_s$ will always be absolutely continuous whenever $a\neq 0$ and $\eta$ is such that $\sigma_\eta^2>0$ or $\nu_\eta(\bR) = \infty$. See also Bodnarchuk and Kulik \cite[Prop. 2, Theorem 1]{BodnarchukKulik2009}, who even characterised when $\smash{\int_0^t \re^{-as} \, \di \eta_s}$ has a bounded density for all $t>0$. Berger \cite[Lemma 4.1]{Berger2018} showed that when $\eta$ has infinite L\'evy measure and $f:[0,t] \to \bR$ is a $C^1$-diffeomorphism onto its range, then $\smash{\int_0^t f(s) \, \di \eta_s}$ is absolutely continuous. Part (iii) below, which essentially is Remark 4.2 in Berger \cite{Berger2018}, generalises this result. For the reader's convenience, we repeat his short proof.

\begin{corollary}[Absolute continuity of $\int f(s) \, \di \eta_s$] \label{t-ac1}
Let $f:[0,\infty) \to \bR$ be a deterministic Borel measurable function such that the improper integral $\int_0^\infty f(s) \, \di \eta_s$ exists. Assume that~$f\neq 0$ on a set of positive Lebesgue measure. Then each of the following conditions implies that $\int_0^\infty f(s) \, \di \eta_s$ is absolutely continuous with respect to Lebesgue measure:
\begin{enumerate}
\item[(i)]  The characteristic triplet of $\eta$ satisfies
\begin{equation} \label{eq-Kallenberg1}
	\lambda^1(\{ s\in [0,\infty) : f(s) \neq 0\})	\cdot \liminf_{\varepsilon \downarrow 0} \varepsilon^{-2} |\ln \varepsilon|^{-1}\left( \sigma_\eta^2 + \int_{-\varepsilon}^{\varepsilon} x^2 \, \nu_\eta (\di x)\right) >\frac14,
		\end{equation}
or more generally, the  characteristic exponent $\Psi_\eta$ satisfies
\begin{equation}
\lambda^1(\{ s\in [0,\infty) : f(s) \neq 0\}) \cdot \liminf_{|z|\to \infty} \frac{-\Re (\Psi_\eta(z))}{\ln (1+|z|)} > \frac12. \label{eq-HW1}
\end{equation}
\item[(ii)] The absolutely continuous part $\nu_{\eta,{\rm{ac}}}$ of $\nu_\eta$ satisfies
$$\lambda^1(\{s\in [0,\infty): f(s) \neq 0\}) \cdot \nu_{\eta,{\rm{ac}}}(\bR) = \infty.$$
\item[(iii)] Preimages of Lebesgue nullsets $B\in \cB_1^*$  under the mapping $f$ are again Lebesgue nullsets (i.e., $\lambda^1(f^{-1} (B)) = 0$ for every $B\in \cB_1^*$ with $\lambda^1(B) = 0$), and
    $$\lambda^1(\{s\in [0,\infty): f(s) \neq 0\}) \cdot \nu_{\eta}(\bR) = \infty.$$
\item[(iv)] There is $b>0$ such that $\eta_b$ is absolutely continuous and the function $f$ is constant and different from zero on an interval of length $b$.
\end{enumerate}
\end{corollary}

The condition that preimages of Lebesgue nullsets under a mapping are again Lebesgue nullsets is called the \emph{Lusin$(N^{-1})$ condition} in the literature. The first condition in Corollary \ref{t-ac1}~(iii) therefore means that $f_{|f^{-1}(\bR^*)} : f^{-1} (\bR^*) \to \bR^*$ satisfies the Lusin$(N^{-1})$ condition. Observe that the Lusin$(N^{-1})$-condition is trivially satisfied if there is a countable decomposition of $[0,\infty)$ into intervals of the form $[a_i,b_i)$ such that each $f_{|{(a_i,b_i)}}$ is a~$C^1$-diffeomorphism onto its range.

\begin{proof}[Proof of Corollary \ref{t-ac1}]
(i) That \eqref{eq-Kallenberg1} implies \eqref{eq-HW1} follows by standard arguments: if ${\sigma_\eta^2>0}$, then this is clear, and if $\sigma_\eta^2=0$ use
$\sin (x) \geq 2x/\pi$ for $x\in [0,\pi/2]$ so that $1-\cos(u)= 2 (\sin (u/2))^2 \geq 2 u^2/\pi^2$ for $|u|\leq \pi$ and hence
$$\frac{-\Re (\Psi_\eta(z))}{\ln (1+|z|)} = \frac{\int_\bR (1-\cos (xz)) \, \nu_\eta(\di x)}{\ln(1+|z|)} \geq 2 \frac{(z/\pi)^2 \int_{-\pi/z}^{\pi/z} u^2\, \nu_\eta(\di u)}{\ln (|z|/\pi)} \frac{\ln (|z|/\pi)}{\ln (1+|z|)} $$
    for $|z|\geq \pi$.

\medskip
    Now assume \eqref{eq-HW1}.
Since $\re^{\Re \Psi_\eta(z)} = \left| \re^{\Psi_\eta(z)}\right| \leq 1$ we have $\Re  \Psi_\eta(z) \leq 0$ for each $z\in \bR$. Let $T>0$ be arbitrary. An application of Fatou's lemma and \eqref{eq-LM2} then shows
\begin{eqnarray*}
\liminf_{|z|\to \infty} \frac{- \Re (\Psi_f (z))}{\ln (1+|z|)} & \geq & \liminf_{|z|\to \infty} \int_0^T \frac{- \Re (\Psi_\eta(f(t) z))}{\ln (1+|z|)} \, \di t \\
& \geq & \int_0^T \liminf_{|z|\to \infty} \frac{- \Re (\Psi_\eta(f(t) z))}{\ln (1+|z|)} \, \di t \\
& = & \int_0^T \one_{\{f(t) \neq 0\}} \, \di t \cdot \liminf_{|z|\to\infty} \frac{- \Re (\Psi_\eta (z))}{\ln (1+|z|)},
\end{eqnarray*}
since $\Psi_\eta(0)=0$. Letting $T\to\infty$ we conclude
$$\liminf_{|z|\to \infty} \frac{- \Re (\Psi_f (z))}{\ln (1+|z|)} \geq \lambda^1(\{ s\in [0,\infty) : f(s) \neq 0\}) \cdot \liminf_{|z|\to \infty} \frac{-\Re (\Psi_\eta(z))}{\ln (1+|z|)}> \frac12$$
by \eqref{eq-HW1}. Hence $\Psi_f$ satisfies \eqref{eq-David2} showing that $\int_0^\infty f(s) \, \di \eta_s$ has a square integrable density.

\medskip
(ii) Denoting
\begin{eqnarray*}
\nu_{f,1} (B) & := & \int_0^\infty \int_\bR \one_{B \setminus \{0\}} (f(t) x) \, \nu_{\eta,{\rm{ac}}}(\di x) \, \di t \quad \mbox{and} \\
\nu_{f,2} (B) & := &  \int_0^\infty \int_\bR \one_{B \setminus \{0\}} (f(t) x) \, \nu_{\eta,{\rm{sing}}}(\di x) \, \di t
\end{eqnarray*}
for $B\in \cB_1$ we have $\nu_f = \nu_{f,1} + \nu_{f,2}$ by \eqref{eq-LM1a}. Now if $B\in \cB_1$ has Lebesgue measure zero, then so has $B/f(s)$ for $f(s) \neq 0$, and it follows
$$\int_{\bR} \one_{B \setminus \{0\}} (f(s) x) \nu_{\eta,{\rm{ac}}}\, (\di x) = \begin{cases}
\nu_{\eta,{\rm{ac}}} ( B/f(s)) = 0, & \mbox{if $f(s) \neq 0$},\\
\int_\bR \one_{B \setminus \{0\}} (0) \, \nu_{\eta,{\rm{ac}}}(\di x)  =0 , & \mbox{if $f(s) = 0$},
\end{cases}
$$ showing that $\nu_{f,1}$ is absolutely continuous. Since $\nu_{f,1}(\bR) = \lambda^1 (\{ s\in [0,\infty): f(s) \neq 0\}) \, \nu_{\eta,{\rm{ac}}}(\bR)=\infty$ by \eqref{eq-LM1a} and assumption, it follows from Lemma \ref{prop-id-ac}~(iii) that each infinitely divisible distribution with L\'evy measure $\nu_{f,1}$ is absolutely continuous. Since $\nu_f = \nu_{f,1} + \nu_{f,2}$, each such distribution (with Gaussian variance zero) is a convolution factor of $\cL(\int_0^\infty f(s) \, \di \eta_s)$, showing that also $\int_0^\infty f(s) \, \di \eta_s$ is absolutely continuous.

\medskip
(iii) For each $B\in \cB_1^*$ with  $\lambda^1(B) = 0$, and each $x\neq 0$, the set $B/x$ is also a Lebesgue nullset and hence so is $f^{-1} (B/x)$ by the stated condition. Using Fubini's theorem and~\eqref{eq-LM1a} we then obtain that
$$\nu_f(B) = \int_\bR \int_0^\infty \one_{B\setminus \{0\}} (f(t) x) \, \di t \, \nu_\eta(\di x) =
 \int_\bR \lambda^1( f^{-1} (B/x)) \, \nu_\eta(\di x) = 0,$$
showing that $\nu_f$ is absolutely continuous. Using $\nu_f(\bR) = \infty$ by assumption the claim then follows again from Lemma \ref{prop-id-ac}~(iii).

\medskip
(iv) Let $f(x) = c \neq 0$ for $x\in (a,a+b)$. Then
$$\int_0^\infty f(s) \, \di \eta_s = \int_0^a f(s) \, \di \eta_s + c (\eta_{a+b} - \eta_a) + \int_b^\infty f(s) \, \di \eta_s.$$
The result then follows by observing that $c (\eta_{a+b} - \eta_a) \stackrel{d}{=} c \eta_{b}$ is absolutely continuous by assumption and independent of $\int_0^a f(s) \, \di \eta_s$ and  $\int_{a+b}^\infty f(s) \, \di \eta_s$.
\end{proof}

\subsection{Conditions for absolute continuity of $\int_0^R f(s) \, \di \eta_s$}
We are also interested in continuity and absolute continuity of randomly stopped functionals such as $\int_0^R f(s) \, \di \eta_s$, where $R$ is an independent time taking values in $(0,\infty)$. Observe that since we have chosen a c\`adl\`ag version of $(\int_0^t f(s) \, \di \eta_s)_{t\geq 0}$, $\int_0^R f(s)\, \di \eta_s$ is a random variable.

\begin{corollary}[Absolute continuity of $\smash{\int_0^R f(s) \, \di \eta_s}$] \label{c-ac1}
Let $f:[0,\infty) \to \bR$ be a Borel measurable deterministic function that is locally integrable with respect to $\eta$ and such that $\lambda^1(\{s\in [0,t]: f(s) \neq 0\}) > 0$ for all $t>0$. Let $R$ be a random variable with values in $(0,\infty)$ that is independent of $\eta$.
Then each of the following conditions implies that $\int_0^R f(s) \, \di \eta_s$ is absolutely continuous:
\begin{enumerate}
\item[(i)] The characteristic triplet of $\eta$ satisfies Kallenberg's condition \eqref{eq-Kallenberg5}, or more generally, the characteristic exponent $\Psi_\eta$ satisfies the Hartman--Wintner condition \eqref{eq-HW3}.
\item[(ii)] $\nu_{\eta,{\rm{ac}}} (\bR) = \infty$.
\item[(iii)] Preimages of Lebesgue nullsets $B\in \cB_1^*$ under the mapping $f$ are again Lebesgue nullsets and $\nu_\eta(\bR) = \infty$.
\item[(iv)] The function $f$ is constant and different from zero in a neighbourhood of zero and $\eta_t$ is absolutely continuous for each $t>0$.
\item[(v)] $\eta$ is of finite variation with non-zero drift, $f$ is Lebesgue almost everywhere different from zero, and $R$ is absolutely continuous with respect to Lebesgue measure.
\end{enumerate}
\end{corollary}

\begin{proof}
Under each of the conditions (i) -- (iv), it follows from Corollary \ref{t-ac1} that $\int_0^t f(s) \, \di \eta_s$ is absolutely continuous for each $t>0$. Conditioning on $R=t$ then gives for each $B\in \cB_1$ with $\lambda^1(B) = 0$ that
$$P\left(\int_0^R f(s) \, \di \eta_s \in B\right) = \int_0^\infty P\left(\int_0^t f(s) \, \di \eta_s \in B \right) P_R(\di t) = 0,$$ showing absolute continuity of $\int_0^R f(s) \, \di \eta_s$.

\medskip
Now assume Condition (v) and denote by $\gamma_\eta^0$ the drift of $\eta$. By interpreting the integral as a pathwise Lebesgue--Stieltjes integral, we can condition on the paths  $(\eta_t)_{t\geq 0} = (g(t))_{t\geq 0}$. Since $f\neq 0$ Lebesgue almost everywhere and since the paths of $\eta$ are of the form $g(t) = \gamma_\eta^0 t + \sum_{0<s\leq t} \Delta g(s)$ with $\lambda^1(s\in(0,t]: \Delta g(s)\neq 0)=0$,
the functions
$H_g: (0,\infty) \to \bR, \quad t\mapsto \int_0^t f(s) \, \di g(s) = \gamma_\eta^0 \int_0^t f(s) \, \di s + \sum_{0<s\leq t} f(s) \Delta g(s)$ are Lebesgue almost everywhere differentiable with derivatives $f(t) \gamma_\eta^0 \neq 0$.  Theorem 4.2 in Davydov et al. \cite{DavydovLifshitsSmorodina1998} shows that the image measure $H_g(\lambda^1)$ is absolutely continuous, for $P_\eta$ almost every path $g$ of $\eta$.
For a Borel set $B$ with $\lambda^1(B) = 0$ we then have $\lambda^1(H_g^{-1} (B)) = 0$ and by absolute continuity of $R$ that $P(H_g(R) \in B) = P (R\in H_g^{-1} (B)) = 0$. Absolute continuity of $\int_0^R f(t) \, \di \eta_t$ then follows from
\begin{eqnarray*}
P\left( \int_0^R f(s) \, \di \eta_s  \in B\right) & = & \int_{D([0,\infty),\bR)} P\left( \int_0^R f(s) \, \di \eta_s \in B \Big| \eta = g\right) \, P_\eta(\di g)   \\
& = & \int_{D([0,\infty),\bR)} P (H_g(R) \in B) \, P_\eta(\di g) = 0
\end{eqnarray*}
for all $B\in \cB_1$ with $\lambda^1(B) = 0$.
\end{proof}


\subsection{Conditions for (absolute) continuity of $\int_0^t \re^{-\xi_{s-}} \, \di \eta_s$} \label{S-t-fixed}

In this section we give sufficient conditions for absolute continuity of $V_{\xi,\eta}(t) := \int_0^t \re^{-\xi_{s-}} \, \di \eta_s$ for fixed $t>0$, where $\xi$ and $\eta$ are independent L\'evy processes. We also characterise in Corollary \ref{char-continuous} when $V_{\xi,\eta}(t)$ is continuous. We start with an example which is due to Lifshits~\cite{Lifshits1982}.

\begin{example} \label{ex-Lifshits}
{\rm Let $\xi$ be a Brownian motion with variance $\sigma_\xi^2>0$ and drift $\gamma_\xi^0 \in \bR$. Then~$\smash{\int_0^t \re^{-\xi_s} \, \di s}$ is absolutely continuous for each $t>0$. This is stated in Problem 9.1 of \cite[p.~53]{DavydovLifshitsSmorodina1998}, but can also be deduced from Theorem 1 in \cite{Lifshits1982}, which shows that random variables of the form $\int_0^t g(\xi_s) \, \di s$ are absolutely continuous, provided $g$ is locally Lipschitz with derivative $g'$ that is non-zero and continuous on some set of full Lebesgue measure, and the autocovariance function $k(s,s') = {\rm{Cov}} (\xi_s,\xi_{s'})$ satisfies the non-degeneracy condition $\int_0^t \int_0^t k(s,s') h(s) h(s') \, \di s \, \di s'>0$ for any Lebesgue integrable function $h:[0,t] \to \bR$ that is not almost everywhere equal to zero. The condition on $g$ is clearly satisfied for $g(x) = \re^{-x}$, and the non-degeneracy condition is equivalent to the fact that $\int_0^t \xi_s h(s)\, \di s$ is not constant for each integrable $h$ that is not almost everywhere equal to zero. To see that the latter condition is satisfied, we can assume without loss of generality that $\sigma_\xi^2=1$ and $\gamma_\xi^0 = 0$. If then $\smash{\int_0^t \xi_s h(s)\, \di s}$ is constant for some integrable function $h$, it must necessarily be equal to its expectation which is zero. Denoting $H(s) = \int_0^s h(u) \, \di u$ and using partial integration we conclude
$$0 = \int_0^t \xi_s \, \di H(s) = \xi_t H(t) - \int_0^t H(s) \, \di \xi_s= \int_0^t (H(t) - H(s)) \, \di \xi_s,$$
where we used that the quadratic covariation of $H$ and $\xi$ is zero ($H$ being of finite variation). Taking the variance of this we see that $\int_0^t (H(t) - H(s))^2 \, \di s = 0$, showing that~$H$ is constant, which in turn implies that $h$ is almost everywhere equal to zero. Hence the non-degeneracy condition is satisfied and \cite[Thm. 1]{Lifshits1982} gives absolute continuity of $\int_0^t \re^{-\xi_s} \, \di s$.}
\end{example}

Although we will not be able to characterise absolute continuity of $V_{\xi,\eta}(t)$ completely, we will give various sufficient conditions that cover many cases of interest in the next theorem.

\begin{theorem}[Sufficient conditions for absolute continuity of $V_{\xi,\eta}(t)$] \label{t-ac2}
	Let $t\in(0,\infty)$ be fixed and assume   that one of the following conditions  is satisfied:
	\begin{enumerate}
		\item[(i)] The characteristic triplet of $\eta$ satisfies
\begin{equation} \label{eq-Kallenberg10}
		\liminf_{\varepsilon \downarrow 0} \varepsilon^{-2} |\ln \varepsilon|^{-1}\left( \sigma^2_\eta + \int_{-\varepsilon}^\varepsilon x^2 \, \nu_\eta(d x)\right) > \frac{1}{4t},
		\end{equation}
or more generally
\begin{equation} \label{eq-HW10}
\liminf_{|z|\to \infty} \frac{- \Re (\Psi_\eta(z))}{\ln (1+|z|)} > \frac{1}{2t}.
\end{equation}
In particular, this is satisfied when $\sigma_\eta^2 > 0$.
\item[(ii)] The absolutely continuous part of $\nu_\eta$ is infinite: $\nu_{\eta,{\rm{ac}}} (\bR) = +\infty$.
\item[(iii)] The characteristic exponent $\Psi_\xi$ of $\xi$ satisfies Condition \eqref{eq-Hawkes} and ${\nu_\eta(\bR) = \infty}$.
\item[(iv)] $\nu_{\xi,{\rm ac}}(\bR) = \nu_\eta(\bR) = \infty$.
\item[(v)] $\eta$ is of finite variation with  non-zero drift and $\xi$ is such that $\sigma_\xi^2>0$ or $\nu_\xi(\bR) = \infty$.
\item[(vi)] $\xi$ is a compound Poisson process and $\eta_s$ is absolutely continuous for all $s>0$.
\end{enumerate}
	Then $\int_0^t \re^{-\xi_{s-}} \, d\eta_s$ is absolutely continuous.
\end{theorem}

\begin{proof}
(i) -- (iii), (vi): Conditioning on the paths $\xi=f$, we obtain for any Borel set $B$ with $\lambda^1(B) = 0$ that
\begin{equation} \label{eq-ac-3}
P\left( \int_0^t \re^{-\xi_{s-}} \, d\eta_s \in B \right) = \int_{D([0,\infty),\bR)} P\left( \int_0^t \re^{-f(s-)} \, d\eta_s \in B\right) \, P_\xi(df).
\end{equation}
Hence absolute continuity of $V_{\xi,\eta}(t)$ will follow if one can show that $P_\xi$-a.s., $\smash{\int_0^t \re^{-f(s-)} \, \di \eta_s}$ is absolutely continuous. But since $\lambda^1( \{ s\in [0,t] : \re^{-f(s-)} \neq 0\}) = t$, this follows from Corollary \ref{t-ac1}~(i)--(iv) and the discussion preceeding Example \ref{ex-Hawkes} regarding the Condition~\eqref{eq-Hawkes} for (iii). For (vi), observe that each path of a compound Poisson process is constant in a neighbourhood of zero.

\medskip
(v) Assume first that $\nu_\xi(\bR) = \infty$. The proof is similar to the proof given in \cite[Thm.~3.9~(b)]{BertoinLindnerMaller2008}, but we give the argument here since that theorem is not directly applicable to our situation. For given $\varepsilon>0$, denote the time of the $i$'th jump of $\xi$ with absolute jump size greater than $\varepsilon$ by $T_i(\varepsilon)$, $i\in \bN$. Define the process $\xi'$ by $\xi_s' = \xi_s  - \sum_{0<u\leq s, |\Delta \xi_u| > \varepsilon} \Delta \xi_u$. Then $\eta$, $\xi'$, $(T_i(\varepsilon))_{i\in \bN}$ and $(\Delta \xi_{T_i(\varepsilon)})_{i\in \bN}$ are all independent by the L\'evy-It\^o decomposition. Now condition first on the set $\{T_2(\varepsilon) \leq t\}$ and then on all quantities present apart the time $T_1(\varepsilon)$ of the first jump of $\xi$ (i.e., condition on $\eta=g$, $\xi'=f$, $T_i(\varepsilon) = t_i$, $i\geq 2$ and $\Delta \xi_{T_i} = y_i$, $i\in \bN$). Conditional on this set and these quantities, we have
\begin{align*}\lefteqn{\int_0^t \re^{-\xi_{s-}} \, \di \eta_s } \\ &= \int_0^{T_1(\varepsilon)} \re^{-f(s-)} \di g(s) + \int_{T_1(\varepsilon)+}^{T_2(\varepsilon)} \re^{-y_1 - f(s-)} \, \di g(s) + \int_{T_2(\varepsilon)+}^t \re^{-h(s-)} \, \di g(s),\end{align*}
where the function $h$ corresponds (on $\{s> T_2(\varepsilon)\}$) to the path of $(\xi_s)_{s> T_2(\varepsilon)}$, which is known under this conditioning. Then $\int_{T_2(\varepsilon)+}^t \re^{-h(s-)} \, \di g(s)$ is constant and since $g$ is of the form $g(s) = \gamma_\eta^0 s + \sum_{0<u\leq s} \Delta g(u)$, where $\gamma_\eta^0$ denotes the drift of $\eta$, the function
$$H: (0,T_2(\varepsilon)] \ni u \mapsto \int_0^u \re^{-f(s-)} \, \di g(s) + \int_{u+}^{T_2(\varepsilon)} \re^{-y_1 - f(s-)} \, \di g(s)$$
is Lebesgue almost everywhere differentiable in $u$ with derivative $\re^{-f(u-)} \gamma_\eta^0 - \re^{-y_1 -f(u-)} \gamma_\eta^0$. Since $y_1\neq 0\neq \gamma_\eta^0$, this derivative is Lebesgue almost everywhere different from zero.
By Theorem 4.2 in Davydov et al. \cite{DavydovLifshitsSmorodina1998}, the image measure under $H$ of the Lebesgue measure on $(0,T_2(\varepsilon)]$ is absolutely continuous, hence so is $H(T_1(\varepsilon))$ since $T_1(\varepsilon)$ is uniformly distributed on $(0,T_2(\varepsilon))$ (e.g., \cite[Prop. 3.4]{Sato2013}). But this shows that
\begin{align*}P&\left( \int_0^t\re^{-\xi_{s-}} \, \di \eta_s \in B \Big| T_2(\varepsilon) \leq t, \eta=g, \xi'=f, T_i=t_i \, (i\geq 2), \, \Delta \xi_{T_i} = y_i \,(i\in \bN) \right)\\ & = 0 \end{align*}
for all Borel sets $B$ with $\lambda_1(B) = 0$. Integrating all conditions out apart from $\{T_2(\varepsilon) \leq  t\}$, we conclude that $P(\smash{\int_0^t\re^{-\xi_{s-}} \, \di \eta_s} \in B | T_2(\varepsilon) \leq t) = 0$. Letting $\varepsilon\downarrow 0$ and observing that $P(T_2(\varepsilon) \leq t) \to 1$ as $\varepsilon\downarrow 0$ as a consequence of $\nu_\xi(\bR) = +\infty$, we conclude $P(\smash{\int_0^t \re^{-\xi_{s-}} \, \di \eta_s} \in B) = 0$ and hence absolute continuity of $V_{\xi,\eta}(t)$.

\medskip
Now assume that $\nu_\xi(\bR) < \infty$ and $\sigma_\xi^2>0$. Since $\xi$ satisfies Condition \eqref{eq-Hawkes}, by (iii) we can additionally assume that $\nu_\eta(\bR) < \infty$. Denote by $M$ the time of the last jump of $\xi$ or $\eta$ before time $t$, i.e., the last time such that neither $\xi$ nor $\eta$ jumps in $(M,t]$ (if no jump occurs, then $M=0$). Then $M<t$ a.s., and conditional on $M=m$, $(\xi_s)_{s\in [0,m]} = (f(s))_{s\in [0,m]}$  and $(\eta_s)_{s\in [0,m]} = (g(s))_{s\in [0,m]}$, we have
$$\int_0^t \re^{-\xi_{s-}} \, \di \eta_s = \int_0^m \re^{-f(s-)} \, \di g(s) + \gamma_\eta^0 \re^{-f(m)} \int_{m+}^t \re^{-(\xi_{s-} - \xi_m}) \, \di s.$$
But the first term is constant and $(\xi_{s-} - \xi_m)_{s\in (m,t]}$ is a Brownian motion with drift $\gamma_\xi^0$ under this conditioning, hence the second term is absolutely continuous by Example \ref{ex-Lifshits}. Hence $V_{\xi,\eta}(t)$ is absolutely continuous under this conditioning, and integrating the conditions out we see that $V_{\xi,\eta}(t)$ is absolutely continuous.

\medskip
(iv)  Choose a set $D\in \cB_1$ with $\nu_{\xi,{\rm{sing}}}(D) = \nu_{\xi,{\rm{ac}}}(\bR \setminus D) = 0$. For each $\varepsilon \in (0,1)$, denote by $R_\varepsilon$ the time of the first jump of $\xi$ with jump size in $D\cap ((-1,-\varepsilon) \cup (\varepsilon, 1))$, and by $Y_\varepsilon$ its jump size. On the set $\{R_\varepsilon < t\}$ we can write
		\begin{eqnarray*}
			\int_0^t \re^{-\xi_{s-}} \, \dx\eta_s
			& = & \int_0^{R_\varepsilon} \re^{-\xi_{s-}} \, \dx\eta_s + \re^{-Y_\varepsilon} \left( \re^{-\xi_{R_\varepsilon-}} \int_{R_\varepsilon+}^t \re^{-(\xi_{s-} - \xi_{R_\varepsilon})} \, \dx\eta_s \right).		\end{eqnarray*}
		Clearly $\re^{-Y_\varepsilon}$ is independent from $(\int_0^{R_\varepsilon} \re^{-\xi_{s-}} \, \dx\eta_s, \re^{-\xi_{R_\varepsilon-}} \int_{R_\varepsilon+}^t \re^{-(\xi_{s-} - \xi_{R_\varepsilon})} \, \dx\eta_s)$. Further, conditioning on $\xi=f$ and $R_\varepsilon$, we see from Proposition \ref{p-cont1} that  $\int_{R_\varepsilon+}^t \re^{-(f(s-) - f(R_\varepsilon))} \, \dx\eta_s$ has no atoms, i.e.,
$$P\left( \int_{R_{\varepsilon}+}^t \re^{-(f(s-) - f(R_\varepsilon))} \,\di \eta_s =b \Big| R_\varepsilon = r \right) = 0$$
for all $b\in \bR$ and $r\in (0,t)$.
Integrating out the condition we see similarly to \eqref{eq-ac-3} that also $\int_{R_\varepsilon+}^t \re^{-(\xi_{s-} - \xi_{R_\varepsilon})} \,\dx\eta_s$ has no atoms when conditioned on the set $\{R_\varepsilon < t\}$. In particular,
		$ \re^{-\xi_{R_\varepsilon-}} \int_{R_\varepsilon+}^t \re^{-(\xi_{s-} - \xi_{R_\varepsilon})} \, \dx\eta_s \neq 0$ a.s. on $\{R_\varepsilon < t\}$. Conditioning on $(\int_0^{R_\varepsilon} \re^{-\xi_{s-}} \, \dx\eta_s,$ $ \re^{-\xi_{R_\varepsilon-}} \int_{R_\varepsilon+}^t \re^{-(\xi_{s-} - \xi_{R_\varepsilon})} \, \dx\eta_s) = (h_1,h_2)$ and observing that $e^{-Y_\varepsilon}$ is absolutely continuous, we see that $h_1 + \re^{-Y_\varepsilon} h_2$ is absolutely continuous. Integrating out $h_1$ and $h_2$, it follows that $\int_0^t \re^{-\xi_{s-}} \, \dx\eta_s$ is absolutely continuous on $\{R_\varepsilon < t\}$ (similar to the proof of (v)). Letting $\varepsilon \downarrow 0$ it follows that $\int_0^t \re^{-\xi_{s-}} \, d\eta_s$ is absolutely continuous since $P(\{R_\varepsilon < t\}) \to 1$, again similar to the proof of (v).
\end{proof}

When $\sigma_\xi^2>0$, we obtain in particular:

\begin{corollary} \label{c-ac2}
Let $\sigma_\xi^2>0$. Then $V_{\xi,\eta}(t)$ is absolutely continuous if and only if $\eta$ is neither the zero process nor a compound Poisson process.
\end{corollary}

\begin{proof}
If $\eta$ is a compound Poisson process, the probability that $\eta$ does not jump before time $t$ is positive, and on this set $V_{\xi,\eta}(t) = 0$, hence $V_{\xi,\eta}(t)$ has an atom at zero and hence is not absolutely continuous. Similarly, if $\eta$ is the zero process, then $V_{\xi,\eta}(t) = 0$. Otherwise, $\eta$ is of finite variation with non-zero drift, or satisfies $\nu_\eta(\bR) = \infty$ or $\sigma_\eta^2>0$. It then follows from Theorem \ref{t-ac2}~(v),(iii),(i) and Example \ref{ex-Hawkes} that $V_{\xi,\eta}(t)$ is absolutely continuous.
\end{proof}

\begin{remark}\label{rem-ext}
{\rm Similar to Remark \ref{rem-ext-new}, conditions (i)-(iii) and (vi) of Theorem \ref{t-ac2} ensure absolute continuity for functionals of the form $\int_0^t g(\xi_{s-}) \, \di \eta_s$ for more general functions $g$ than the exponential function. To be more precise, let $g:[0,t] \to \bR$ be continuous, such that $s\mapsto g(\xi_{s-})$ is c\`agl\`ad.\\
(i) Assume that  $g(0) \neq 0$ and that  Condition (vi) of Theorem \ref{t-ac2} is satisfied.  Then $\smash{\int_0^t g(\xi_{s-}) \, \di \eta_s}$ is absolutely continuous by the same proof.\\
(ii) Assume that one of the conditions (i) - (iii) of Theorem \ref{t-ac2} is satisfied, denote by $N$ the zero set of $g$, and assume \eqref{eq-g-N}, i.e., $\smash{\int_0^t P(\xi_{s} \in N) \, \di s = 0}$. Then, using Fubini's theorem, we conclude that $E \int_0^t \one_N(\xi_{s}) \, \di s=0$, showing that $\lambda^1(\{ s\in [0,t]: \xi_s \in N\}) = \lambda^1 (\{s\in [0,t] : g(\xi_s) = 0\}) = 0$ a.s. The result then follows by the same proof as in Theorem \ref{t-ac2}~(i)--(iii) and observing that $\int_0^t g(\xi_{s-}) \, \di \eta_s$ and $\int_0^t g(\xi_{s}) \, \di \eta_s$ are a.s. equal.}
\end{remark}

We can now characterise continuity of $V_{\xi,\eta}(t)$.

\begin{corollary}[Continuity of $V_{\xi,\eta}(t)$] \label{char-continuous}
$V_{\xi,\eta}(t)$ has atoms (i.e., is not continuous) if and only if $\eta$ is the zero process, or $\eta$ is a compound Poisson process, or
\begin{equation} \label{eq-cont1}
\sigma_\eta^2 = \sigma_\xi^2 = 0,\quad \nu_\eta(\bR) < \infty \quad \mbox{and} \quad  \nu_\xi(\bR) < \infty.
\end{equation}
\end{corollary}

\begin{proof}
As seen in the proof of Corollary \ref{c-ac2}, if $\eta$ is the zero process or a compound Poisson process, then $V_{\xi,\eta}(t)$ has an atom at zero and  hence is not continuous. If \eqref{eq-cont1} holds but $\eta$ is neither a compound Poisson process nor the zero process, then $\eta$ has drift $\gamma_\eta^0 \neq 0$ and $\xi$ is of finite variation and finite jump activity with drift $\gamma_\xi^0 \in \bR$. The probability that both $\eta$ and $\xi$ do not jump before time $t$ is positive, and on this set we have $\smash{V_{\xi,\eta}(t)  = \gamma_\eta^0 \int_0^t \re^{-\gamma_\xi^0 s} \, \di s,}$ so that $V_{\xi,\eta}(t)$ has an atom.

\medskip
Now assume that neither \eqref{eq-cont1} is satisfied nor that  $\eta$ is a compound Poisson process nor the zero process. If $\sigma_\eta^2 > 0$ or $\nu_\eta(\bR) = \infty$, conditioning on
 $\xi=f$ we see that~$\smash{\int_0^t \re^{-f(s-)}\, \di \eta_s}$ is continuous by Proposition \ref{p-cont1}. Hence
 $$P(V_{\xi,\eta}(t) = b) = \int_{D([0,\infty),\bR)} P(V_{\xi,\eta}(t) = b|\xi=f) \, P_\xi(\di f) = 0$$
for each $b\in \bR$, so that $V_{\xi,\eta}(t)$ is continuous. If $\nu_\eta(\bR) < \infty$ and $\sigma_\eta^2=0$, we must have $\gamma_\eta^0 \neq 0$ since $\eta$ is not a compound Poisson process and not the zero process. Since \eqref{eq-cont1} is violated, necessarily $\sigma_\xi^2>0$ or $\nu_\xi(\bR) = \infty$. Then $V_{\xi,\eta}(t)$ is absolutely continuous and hence continuous by Theorem \ref{t-ac2}~(v).
\end{proof}

\begin{remark}
{\rm \begin{enumerate}
\item[(i)]
It is clear that a pure types theorem does not hold for the law of $\int_0^t \re^{-\xi_{s-}} \, \di \eta_s$. The simplest counterexample is when $\xi=0$ and $\eta$ is a compound Poisson process with absolutely continuous jump distribution. Then $V_{\xi,\eta}(t)=\eta_t$ whose distribution has an atom at zero, but restricted to $\bR^*$ has a density. Similar examples can be constructed when both $\xi$ and $\eta$ are compound Poisson processes.
\item[(ii)] An example when $\int_0^t \re^{-\xi_{s-}} \, \di \eta_s$ is continuous singular is easily constructed by choosing $\xi=0$ and for $\eta$ a process for which $\eta_t$ is continuous singular, examples of which are given in \cite[Thms. 27.19, 27.23]{Sato2013}.
\end{enumerate} }
\end{remark}


\section{Proofs for the results in Section \ref{S5neu}}\label{S5newproofs}

\begin{proof}[Proof of Theorem \ref{t-ac3}]
	(i) -- (iv):	By Theorem  \ref{t-ac2}, any of the given conditions (i) to (iv)  implies absolute continuity of $V_{\xi,\eta}(t) = \int_0^t e^{-\xi_{s-}} \, d\eta_s$ for all $t>0$. Let $B\in \cB_1$ be a Lebesgue-null set. Conditioning on $\tau=t$ we obtain
		\begin{align*} P \left(\int_0^\tau \re^{-\xi_{s-}} \, d\eta_s \in B\right) & = \int_0^\infty P \left(\int_0^t \re^{-\xi_s} \, d\eta_s \in B \Big| \tau = t\right) \, P_\tau (\di t)\\ & = \int_0^\infty 0 \, P_\tau (\di t) = 0,\end{align*} 
	showing that $\int_0^\tau e^{-\xi_{s-}} \, d\eta_s$ is absolutely continuous.

\medskip
	(v) If $\eta$ is of finite variation with non-zero drift, we can condition on the path $\xi=f$. By Corollary \ref{c-ac1}~(v), $\int_0^\tau \re^{-f(s-)} \di \eta_s$ will be absolutely continuous for each path $f$. Integrating the condition out we see as in  \eqref{eq-ac-3} that $\int_0^\tau \re^{-\xi_{s-}} \, \di \eta_s$ is absolutely continuous.

\medskip
	(vi) Denote by $T$ the time of the first jump of $\xi$. This is exponentially distributed with parameter $\nu_\xi(\bR) \in (0,\infty)$ and independent from $\tau$. Then, conditional on $\{\tau > T\}$, the random variables $\tau-T$ and $T$ are conditionally independent and (conditionally) exponentially distributed with parameters $q$ and $q+\nu_\xi(\bR)$, respectively. Conditional on~$\{\tau > T\}$ we have
	$$V_{q,\xi,\eta} = \int_0^T \di \eta_s + \int_{T+}^\tau \re^{-\xi_{s-}} \, \di \eta_s = \eta_T + \int_{T+}^{T + (\tau-T)} \re^{-\xi_{s-}} \, \di \eta_s.$$
	Since the two summands are conditionally independent and the first is absolutely continuous by~\eqref{eq-ACP} as $T$ is conditionally exponentially distributed with parameter $q+\nu_\xi(\bR)$, the distribution of $V_{q,\xi,\eta}$ conditional on $\{\tau > T\}$ is absolutely continuous. On the other hand, conditional on $\{\tau < T\}$ we have $V_{q,\xi,\eta} = \eta_\tau$ with $\tau$ being conditionally exponentially distributed with parameter $q+\nu_\xi(\bR)$, so that~\eqref{eq-ACP} gives absolute continuity of $V_{q,\xi,\eta}$ also under this conditioning. Adding up the two cases shows absolute continuity of $V_{q,\xi,\eta}$.
\end{proof}

\begin{proof}[Proof of Corollary \ref{c-ac4}]
	If $\eta$ is a compound Poisson process, then the probability that $\eta$ does not jump before time $\tau$ is positive, hence $V_{q,\xi,\eta}$ has an atom at zero and is hence not continuous, and similarly when $\eta$ is the zero process. If $\eta$ is neither a compound Poisson process nor the zero process, it is of finite variation with non-zero drift, or satisfies $\nu_\eta(\bR) = \infty$ or $\sigma_\eta^2>0$, in which case $V_{q,\xi,\eta}$ is absolutely continuous by Theorem \ref{t-ac3}~(v), (iii), or (i), in combination with Example \ref{ex-Hawkes}.
\end{proof}

\begin{proof}[Proof of Corollary \ref{c-ac7}]
	Sufficiency of the ACP condition follows from Theorem \ref{t-ac3}~(vi). For the converse, assume that $V_{q,\xi,\eta}$ is absolutely continuous and denote by $T$ the time of the first jump of $\xi$. Then also conditional on $\{\tau < T\}$, $V_{q,\xi,\eta}=\eta_\tau$ is absolutely continuous. But $\tau$ is conditional exponentially distributed with parameter $q+\nu_\xi(\bR)$, showing that $\eta$ satisfies the ACP condition by the discussion following Example \ref{ex-Hawkes}.
\end{proof}

\begin{proof}[Proof of Proposition \ref{c-ac5}]
	We have already seen in the proof of Corollary \ref{c-ac4} that if $\eta$ is a compound Poisson process or the zero process, then $V_{q,\xi,\eta}$ has an atom at zero and is hence not continuous. If $\eta$ is neither a compound Poisson process nor the zero process and does not satisfy Condition \eqref{eq-cont1}, then $P(V_{\xi,\eta}(t) = b) = 0$ for any $t>0$ and $b\in \bR$ by Corollary \ref{char-continuous}. Continuity of $V_{q,\xi,\eta}$ then follows by conditioning on $\tau=t$ via
	$$P(V_{q,\xi,\eta} = b) = \int_0^\infty P \left( V_{\xi,\eta}(t) = b \big| \tau = t\right) P_\tau (\di t) = \int_0^\infty 0 \, P_\tau (\di t) = 0.$$
	Finally, if $\eta$ is neither a compound Poisson process nor the zero process but satisfies Condition \eqref{eq-cont1}, then it must be of finite variation with non-zero drift, so that $V_{q,\xi,\eta}$ is absolutely continuous by Theorem \ref{t-ac3}~(v).
\end{proof}

\begin{proof}[Proof of Theorem \ref{t-ac5}]
	(viii) As mentioned before, $V_{0,\xi,\eta}$ is self-decomposable when $\xi$ is spectrally negative. Since additionally  it is not constant a.s. if additionally at least one of $\xi$ and $\eta$ is non-deterministic by \cite[Thm. 2.2]{BertoinLindnerMaller2008}, it is absolutely continuous in this case (\cite[Ex. 27.8]{Sato2013}).

\medskip
	(v) Let $\eta$ be of finite variation with non-zero drift. Since $\int_0^\infty \re^{-\xi_{s-}} \, \di \eta_s$ converges, $\xi$ must be transient. It then follows from \cite[Thm. 3.9(a)]{BertoinLindnerMaller2008} that $V_{0,\eta,\xi}$ is absolutely continuous when $\nu_\xi (\bR) > 0$. If $\nu_\xi(\bR) = 0$, then $\xi$ is spectrally negative, and absolute continuity of $V_{0,\xi,\eta}$ follows from (viii).

\medskip
	(i) -- (iii): As long as $\nu_\eta(\RR)>0$ in (iii), this follows in complete analogy to the corresponding proof of Theorem \ref{t-ac2}~(i)--(iii) by conditioning on the paths $\xi=f$, setting $t=\infty$, and observing that $\lambda^1(\{ s \in [0,\infty) : \re^{-f(s-)} \neq 0\}) = \infty$. If $\nu_\eta(\RR)=0$ in (iii), either $\sigma_\eta^2>0$ or $\eta$ is deterministic (but non-zero) and hence of finite variation. The claim then follows from the previously shown cases (i) and (v).

\medskip
	(iv)  Choose a set $D\in \cB_1$ that is bounded away from zero with $\nu_{\xi,{\rm{ac}}}(D) > 0$ and $\nu_{\xi,{\rm{sing}}}(D) = 0$.
	Denote by $R$ the time of the first jump of $\xi$ with jumping size in $D$ and by~$Y$ its jump size (which has a density by assumption). Observe that $R$ is a stopping time with respect to the augmented filtration. Writing
		$$\int_0^\infty \re^{-\xi_{s-}} \, \dx\eta_s = \int_0^R \re^{-\xi_{s-}} \, \dx\eta_s + \re^{-Y} \left( \re^{-\xi_{R-}} \int_{R+}^\infty \re^{-(\xi_{s-} - \xi_R)} \, \dx\eta_s\right),$$
	observing that $Y$ is independent from
	$(\int_0^R \re^{-\xi_{s-}} \, \dx\eta_s,$  $\re^{-\xi_{R-}} \int_{R+}^\infty \re^{-(\xi_{s-} - \xi_R)} \, \dx\eta_s)$ and that $\int_{R+}^\infty \re^{-(\xi_{s-} - \xi_R)} \, \dx\eta_s \stackrel{d}{=} \int_0^\infty \re^{-\xi_{s-}} \, \dx\eta_s$ is continuous (hence $\neq 0$ a.s.) by Theorem 2.2 in \cite{BertoinLindnerMaller2008}, it follows as in the proof of Theorem \ref{t-ac2}~(iv) that $\int_0^\infty \re^{-\xi_{s-}} \, \dx\eta_s$ is absolutely continuous.

\medskip
	(vi) Denote by $T$ the time of the first jump of $\xi$, which is exponentially distributed with parameter $\nu_\xi(\bR) \in (0,\infty)$. Then
		$$V_{0,\xi,\eta} = \int_0^T \di \eta_s + \int_{T+}^\infty \re^{-\xi_{s-}} \, \di \eta_s = \eta_T +  \re^{-\xi_T} \int_{T+}^\infty \re^{-(\xi_{s-}-\xi_T)} \, \di \eta_s$$
	with  $\eta_T$, $\xi_T = \Delta \xi_T$ and $\int_{T+}^\infty \re^{-(\xi_{s-}-\xi_T)} \, \di \eta_s$ being independent by the strong Markov property.  Since $\eta_T$ is absolutely continuous by the ACP condition, absolute continuity of $V_{0,\xi,\eta}$ follows.

\medskip
	(vii) Denote by $R_i$ the time of the $i$'th jump of $\eta$ and by $Z_i$ its jump size. Since $\xi$ and~$\eta$ a.s. do not jump together, we have
		$$V_{0,\xi,\eta} = \int_0^\infty \re^{-\xi_{s-}} \, \di \eta_s = \sum_{i=1}^\infty \re^{-\xi_{R_i-}} Z_i \stackrel{a.s.}{=} \re^{-\xi_{R_1}} \left( \sum_{i=1}^\infty \re^{-(\xi_{R_i} - \xi_{R_1})} Z_i \right).$$
	But $\xi_{R_1}$ and $\sum_{i=1}^\infty \re^{-(\xi_{R_i} - \xi_{R_1})} Z_i$ are independent by the strong Markov property, and since~$V_{0,\xi,\eta}$ is different from zero a.s. (since it has no atom as $\eta$ is non-deterministic), also~$\sum_{i=1}^\infty \re^{-(\xi_{R_i} - \xi_{R_1})} Z_i$ is different from zero a.s. The claim then follows by observing that~$\re^{-\xi_{R_1}}$ is absolutely continuous by the ACP condition, since $R_1$ is exponentially distributed with parameter $\nu_\eta(\bR) \in (0,\infty)$.
\end{proof}

\begin{proof}[Proof of Corollary \ref{c-ac6}]
If $\sigma_\eta^2 > 0$, this follows from Theorem \ref{t-ac5}~(i). If $\sigma_\xi^2>0=\sigma^2_\eta$, then~$\xi$ satisfies
Equation \eqref{eq-Hawkes} by Example \ref{ex-Hawkes}, and $\nu_\xi(\bR) >0$ or $\eta$ is deterministic. The claim now follows from parts (iii) and (v) of Theorem \ref{t-ac5}, respectively.
\end{proof}

\section{Proofs for the results in Section \ref{S4}}\label{S4proofs}

In the proof of Theorem~\ref{thm-support}, we will frequently make use of the following three lemmas.

\begin{lemma}\label{support-killing-from-nokilling}
	Let $\xi$ and $\eta$ be 
	such that $V_{0,\xi,\eta}:=\int_0^\infty \re^{-\xi_{s-}} d\eta_s$ exists a.s. Then
		$$\supp(V_{0,\xi,\eta})\subseteq \supp(V_{q,\xi,\eta}) \quad \text{for all }q>0.$$
\end{lemma}
\begin{proof}
	This is clear as for any $q>0$ the random variable $\tau$ may become arbitrarily large.
\end{proof}

\begin{lemma}\label{support-killing-from-lessjumps}
	Let $B_\xi$, $B_\eta\subset \RR$ be two Borel sets that are bounded away from zero. Define $\smash{\widehat{\nu}_{\xi}:=\nu_\xi|_{B_\xi^c}}$ and $\smash{\widehat{\nu}_{\eta}:=\nu_{\eta}|_{B_\eta^c}}$ as the restrictions of $\nu_\xi$ and~$\nu_{\eta}$ on $\RR\setminus B_\xi$ and $\RR\setminus B_\eta$, respectively, and denote
	by $\smash{\widehat{\xi}}$ and $\widehat{\eta}$ the independent L\'evy processes
	\begin{displaymath}
			\widehat{\xi}_t=\xi_t-\sum_{\substack{0<s\leq t\\ \Delta\xi_s\in B_{\xi}}}\Delta\xi_s,\quad \widehat{\eta}_t=\eta_t-\sum_{\substack{0<s\leq t\\ \Delta\eta_s\in B_{\eta}}}\Delta\eta_s
	\end{displaymath}
	with L\'{e}vy measures $\widehat{\nu}_\xi$ and $\widehat{\nu}_\eta$, respectively. Then
		$$\supp(V_{q,\smash{\widehat{\xi}}, \widehat{\eta}})\subseteq \supp(V_{q,\xi,\eta}).$$
\end{lemma}
\begin{proof}
	With positive probability we have $\{\xi_t,\eta_t, 0\leq t\leq \tau\}= \{\smash{\widehat{\xi}}_t, \widehat{\eta}_t, 0\leq t\leq \tau\}$ for any realization of $\tau$. Hence, the claim follows.
\end{proof}

Recall from Section \ref{S2} that for a function $f:[0,t] \to \bR$ which is integrable in the sense of Rajput and Rosinski \cite[p.460]{RajputRosinski1989}, the random variable $\int_0^t f(s) \, \di \eta_s$ is infinitely divisible with characteristic triplet $(\sigma_{f,t}^2,\nu_{f,t},\gamma_{f,t})$ given by \eqref{eq-LM7} -- \eqref{eq-LM5}. From this we derive the following.

\begin{lemma}\label{lemma-integral-id}
	Let ${f\negmedspace:[0,t]\rightarrow\RR}$ be bounded and Borel measurable such that $f$ is not Lebesgue almost everywhere equal to~zero. Then the following are true:
	\begin{itemize}
		\item[(i)] The L\'evy process with characteristic triplet $(\sigma_{f,t}^2,\nu_{f,t},\gamma_{f,t})$ is of finite variation if and only if $\eta$ is of finite variation.
		In that case, the corresponding drifts $\gamma_{f,t}^0$ and $\gamma_\eta^0$ are related by $\gamma_{f,t}^0=\gamma_{\eta}^0\int_0^t f(s)\dx s$.
		\item[(ii)] If $0\in\supp(\nu_{\eta})$, then $0\in\supp(\nu_{f,t})$.
		\item[(iii)] $\nu_{f,t}$ is infinite if and only if $\nu_\eta$ is infinite.
	\end{itemize}
	If, additionally, $f$ is strictly positive on $[0,t]$, then
	\begin{itemize}
		\item[(iv)] $\nu_{f,t}((0,\infty))>0$  if and only if  $\nu_\eta((0,\infty))>0$, and similarly for $(-\infty,0)$.
	\end{itemize}
\end{lemma}

\begin{proof}
	Parts (ii), (iii) and (iv) follow directly from \eqref{eq-LM1}. For the proof of (i),  by measure theoretic induction, i.e., considering linear combinations and limits of indicator functions $g(x)=\mathbf{1}_A(x)$ of Borel sets $A\not\ni 0$, it follows from \eqref{eq-LM1} that
	\begin{equation} \label{eq-LM4}
			\int_\bR g(x) \, \nu_{f,t}(\di x) = \int_0^t \int_\bR g(f(s) x) \, \nu_\eta(\di x) \, \di s
	\end{equation}
	for any Borel measurable function $g:\bR \to [0,\infty)$ satisfying $g(0) = 0$, and similarly for any Borel measurable function $g:\bR\to \bR$ with $g(0) = 0$ for which the integrals exist. Applying
	\eqref{eq-LM4} to~${g(x)=|x|\one_{[-1,1]}(x)}$ then gives
	\begin{align*}
			\int_{-1}^1|y|\nu_{f,t} (\dx y)&
			=\int_0^t|f(s)|\int_{-1/|f(s)|}^{1/|f(s)|}|x|\, \nu_\eta(\dx x)\, \dx s,
	\end{align*}
	from which in combination with \eqref{eq-LM7} the \lq\lq only if\rq\rq~part readily follows. For the converse, observe that the right-hand side of the above equation can be bounded by
	$$\int_0^t |f(s)| \di s \int_{-1}^1 |x|\, \nu_\eta(\di x) + \int_0^t |f(s)| \one_{\{ |f(s)| < 1\}}  \frac{1}{|f(s)|} \nu_\eta(\bR \setminus [-1,1])  \, \di s,$$
	which is finite if $\eta$ is of finite variation since we assumed boundedness of $f$.
	Finally, the expression for the drift $\gamma_{f,t}^0$  follows from \eqref{eq-LM5}, \eqref{eq-LM4} and the fact that ${\gamma_{f,t}^0 = \gamma_{f,t} - \smash{\int_{-1}^1 y \, \nu_{f,t}(\di y)}}$ and similarly for $\gamma_\eta^0$. Observe that (i) could have similarly been derived by a direct application of Theorem 2.10 and Equation (2.16) in \cite{Sato2007}.
\end{proof}

We can now prove Theorem \ref{thm-support}.

\begin{proof}[Proof of Theorem \ref{thm-support}]
	
	(i) In the case $\eta\equiv0$, the integrator induces the zero measure, yielding $V=0$ and thus $\supp(V)=\{0\}$.
	\medskip
	
	(iii)~(a,b)
	First, assume that $\eta$ is of infinite variation. By conditioning on $\tau$ and $\xi$, as mentioned in the preliminaries, we find
	\begin{align}
			P( V\in  B) = \int_0^{\infty}\int_{D([0,\infty),\RR)} P\Big( \int_0^t \re^{-f(s-)} \, d\eta_s \in B\Big) \, P_\xi(\di f)P_{\tau}(\dx t),\label{conditioning}
	\end{align}
	for all $B\in \cB_1$. From Lemma~\ref{lemma-integral-id} it follows that $\smash{\int_0^t\re^{-f(s-)}\dx\eta_s}$ is infinitely divisible with ${\smash{\int_{-1}^1|x|\nu_{f,t}(\dx x)}=\infty}$ or~$\sigma_{f,t}^2>0$. Thus, \cite[Thm.~24.10(i)]{Sato2013} implies ${P(\smash{\int_0^t\re^{-f(s-)}\dx\eta_s}\in B)>0}$ for all open~${B\subseteq\RR}$. Together with~\eqref{conditioning}, this yields $\supp(V)=\RR$.
	
	\medskip
	If $\eta$ is of finite variation with $0\in\supp(\nu_{\eta})$ and $\nu_{\eta}(\RR_+),\nu_{\eta}(\RR_-)>0$, a similar argument is applicable. Conditioning as in~\eqref{conditioning}, we find by Lemma~\ref{lemma-integral-id} that in this case~${0\in\supp(\nu_{f,t})}$, as well as $\nu_{f,t} (\RR_+),\nu_{f,t} (\RR_-)>0$. Therefore, $P(\smash{\int_0^t\re^{-f(s-)}\dx\eta_s}\in B)>0$ for all open $B\subseteq\RR$ by \cite[Thm.~24.10(ii)]{Sato2013} and thus $\supp(V)=\RR$ as claimed.
	\medskip
	
	(iv) when $0\in \supp(\nu_\eta)$:\\
	Let $\eta$ be as in (iv) and assume additionally that zero is in the support of $\nu_\eta$. Assume further that $\supp(\nu_{\eta})\subseteq[0,\infty)$, the case $\supp(\nu_\eta) \subseteq (-\infty,0]$ following by symmetry. Conditioning as in~\eqref{conditioning}, Lemma~\ref{lemma-integral-id} yields that $\smash{\int_0^t\re^{-f(s-)}\dx\eta_s}$ (more precisely, the L\'{e}vy process corresponding to it) is of finite variation with $0\in\supp(\nu_{f,t})$, $\nu_{f,t}(\bR_+) > 0 = \nu_{f,t}(\bR_-)$ and drift $\gamma_\eta^0 \smash{\int_0^t \re^{-f(s-)} \, \di s}$. By \cite[Thm.~24.10(iii)]{Sato2013} we find that
	\begin{equation}\label{support-conditioned-integral}
			\supp\Big(\int_0^t\re^{-f(s-)}\dx\eta_s\Big)=\Big[\gamma_{\eta}^0\int_0^t\re^{-f(s-)}\dx s,\infty\Big).
	\end{equation}
	Since zero is in the support of $\tau$, Equation~\eqref{conditioning} shows that $\supp(V) \supseteq [0,\infty)$. Since $V\geq 0$ a.s. when $\eta$ is a subordinator, we also get the reverse inequality, so that $\supp(V) = [0,\infty)$ when $\eta$ is a subordinator.

\medskip
	If $\eta$ is not a subordinator, by assumption we necessarily have $\gamma_\eta^0 < 0$. If then $\xi$ is a subordinator with strictly positive drift $\gamma_\xi^0$, then $1/\gamma_\xi^0 \in \supp (\smash{\int_0^\infty \re^{-\xi_{s-}} \, \di s})$ by \cite[Lem. 1]{BehmeLindnerMaejima2016} and \eqref{conditioning} and \eqref{support-conditioned-integral} show $\supp(V) \supseteq [\gamma_\eta^0/\gamma_\xi^0, \infty)$, as $\tau$ may become arbitrary large. On the other hand, \begin{equation} \label{eq-estimate2}
			V_{q,\xi,\eta} \geq \gamma_\eta^0 \int_0^\tau \re^{-\xi_{s-}} \, \di s \geq \gamma_\eta^0 \int_0^\infty \re^{-\gamma_\xi^0 s} \, \di s = \gamma_\eta^0/\gamma_\xi^0,
	\end{equation}
	showing the converse inequality when $\xi$ is a subordinator with strictly positive drift.

\medskip
	Finally, assume that $\eta$ is not a subordinator (hence $\gamma_\eta^0<0$) and $\xi$ is not a subordinator with strictly positive drift. If $\xi$ does not drift a.s. to infinity, then ${\int_0^\infty \re^{-\xi_{s-}} \, \di s = +\infty}$ by~\cite[Thm. 2]{EricksonMaller2005}, and if $\xi$ drifts a.s. to infinity, then $\smash{\int_0^\infty \re^{-\xi_{s-}} \, \di s}$ is finite but unbounded by \cite[Lem. 1]{BehmeLindnerMaejima2016}. As $\tau$ may become arbitrarily large, we conclude in both cases from \eqref{conditioning} and~\eqref{support-conditioned-integral} that $\supp(V) = \bR$ in this case.
	\medskip

	(ii), (iii)~(c), (iv) when $\gamma_\eta^0\neq 0$:\\
	Let $\eta$ be of finite variation with non-zero drift $\gamma_\eta^0$. By the cases already proved we can additionally  assume that $0 \not\in \supp(\nu_\eta)$, in particular $\nu_\eta(\bR) < \infty$. By symmetry, we can further assume without loss of generality that $\gamma_\eta^0>0$. We distinguish in the following whether $\xi$ is a subordinator with strictly positive drift or not.

\medskip
	Case 1: Assume $\xi$ is a subordinator with strictly positive drift $\gamma_\xi^0$.\\
	If $\nu_\eta(\bR_+) = 0$, then
	$V_{q,\xi,\eta} \leq {\gamma_\eta^0}/{\gamma_\xi^0}$ by the same estimates that lead to \eqref{eq-estimate2}, so that $\supp(V) \subseteq (-\infty, \gamma_\eta^0/\gamma_\xi^0]$ in this case. If additionally $\nu_\eta(\bR_-) > 0$, choose a constant $K>0$ such that $\nu_\eta([-K,0))>0$ and construct  $\widehat{\eta}$ from $\eta$ by subtracting the jumps that are less than $-K$ of $\eta$. Then $V_{0,\xi,\widehat{\eta}}$ exists by \cite[Thm. 2]{EricksonMaller2005} and from \cite[Thm. 1(iii)]{BehmeLindnerMaejima2016} together with Lemmas \ref{support-killing-from-nokilling} and \ref{support-killing-from-lessjumps} we obtain
		$$\supp(V_{q,\xi,\eta}) \supseteq \supp (V_{q,\xi,\widehat{\eta}}) \supseteq \supp (V_{0,\xi,\widehat{\eta}}) = (-\infty,\gamma_\eta^0/\gamma_\xi^0],$$
	which together with the estimate above gives the desired $\supp(V) = (-\infty,\gamma_\eta^0/\gamma_\xi^0]$, so that parts of (iv) are proved.

\medskip
	If $\nu_\eta(\bR_+) = \nu_\eta(\bR_-) = 0$ (as in (ii)), then $\eta$ is deterministic, $V_{0,\xi,\eta}$ converges and, unless also $\xi$ is deterministic,
		$$\supp(V_{q,\xi,\eta}) \supseteq \supp (V_{0,\xi,\eta}) = [0, \gamma_\eta^0/\gamma_\xi^0]$$
	by \cite[Lem. 1]{BehmeLindnerMaejima2016} and Lemma \ref{support-killing-from-nokilling}, so that together with the previous upper bound we obtain $\supp (V) = [0,\gamma_\eta^0/\gamma_\xi^0]$. If also $\xi_t=\gamma_\xi^0 t$ is deterministic, then ${\gamma_\eta^0 \smash{\int_0^\tau \re^{-\gamma_\xi^0 s}\, \di s = (1-\re^{-\tau \gamma_\xi^0})\gamma_\eta^0/\gamma_\xi^0}}$, giving the same support also in this case.
	
\medskip
	If $\nu_\eta(\bR_+) > 0 = \nu_\eta(\bR_-)$ (as in parts of (iv)), then $\eta$ is a subordinator and hence $V\geq 0$. Choose $K>0$ such that $\nu_\eta((0,K])>0$ and construct $\widehat{\eta}$ from $\eta$ by subtracting the jumps from $\eta$ that are greater than $K$. Again, $V_{0,\xi,\widehat{\eta}}$ exists and its support is $[0,\infty)$ by~\cite[Thm.~1(ii)]{BehmeLindnerMaejima2016} unless $\xi$ is deterministic, so that $\supp(V) = [0,\infty)$ in this case. If $\xi_t = \gamma_\xi^0 t$ is deterministic, then $\supp(V) \supseteq \supp(V_{0,\xi,\eta} ) = [\gamma_\eta^0/\gamma_\xi^0,\infty)$ by \cite[Thm. 1(ii)]{BehmeLindnerMaejima2016}. On the other hand, the assumption $0 \not\in \supp(\nu_\eta)$ implies $\nu_\eta(\bR) < \infty$, so that with positive probability, $\eta$ does not jump before time $\tau$. This implies $\smash{\supp(V_{q,\gamma_\xi^0 t,\eta}) \supseteq \supp (V_{q,\gamma_\xi^0 t, \gamma_\eta^0 t}) = [0, \gamma_\eta^0/\gamma_\xi^0]}$, so that altogether, $\supp(V) = [0,\infty)$ also when $\xi$ is deterministic.

\medskip
	If $\nu_\eta(\bR_+) , \nu_\eta(\bR_-) > 0$ (as in (iii)~(c)), choose $K>0$ such that $\nu_\eta([-K,0))$, $\nu_\eta([0,K]) >0$ and construct $\widehat{\eta}$ from $\eta$ by subtracting the jumps of absolute size greater than $K$. Then $\supp(V_{0,\xi,\widehat{\eta}}) = \bR$ by \cite[Thm. 1(i)]{BehmeLindnerMaejima2016} and $\supp(V) = \bR$ follows as before using Lemmas \ref{support-killing-from-nokilling} and \ref{support-killing-from-lessjumps}.

\medskip
	Case 2: Assume $\xi$ is not a subordinator with strictly positive drift.\\
	If $\xi_t$ does not drift a.s. to $\infty$, then $\smash{\int_0^\infty \re^{-\xi_{s-}} \, \di s} = \infty$ a.s. (cf. \cite[Thm.~2]{EricksonMaller2005}), and if $\xi_t$ drifts a.s. to $\infty$, then $\smash{\int_0^\infty \re^{-\xi_{s-}} \, \di s}$ is finite a.s., but has unbounded support by \cite[Lem.~1]{BehmeLindnerMaejima2016}. In both cases, for any $C>0$ there is some $t(C) \geq 0$ such that $\int_0^{t(C)} \re^{-\xi_{s-}} \, \di s \geq C$ on a set $\Omega_C$ with positive probability. Since ${t\mapsto \smash{\int_0^t \re^{-\xi_{s-}} \, \di s}}$ is pathwise continuous and increasing, on the set $\Omega_C$ it takes all values in $[0, C]$ when $t$ runs through $[0,t(C)]$. Since $\tau$ is independent of $(\xi,\eta)$ and has a strictly positive density on $[0,\infty)$, it follows that $\supp (\smash{\int_0^\tau \re^{-\xi_{s-}} \, \di s})  \supset [0,C]$ for each $C>0$ and hence $\supp(\smash{\int_0^\tau \re^{-\xi_{s-}} \, \di s}) = [0,\infty)$. This finishes the proof of (ii).

\medskip
	If $\nu_\eta(\bR_-) > 0 = \nu_\eta(\bR_+)$ (as in parts of (iv)), choose $a<0$ such that $\nu_\eta((2a,a)) > 0$. Choose $\varepsilon \in (0,1)$ such that $\sup_{s\in [0,\varepsilon]} |\xi_s| \leq 1$ has positive probability (possible since $\xi$ has c\`adl\`ag paths). Then for each $n\in \bN$, the probability that $\sup_{s\in [0,\varepsilon]} |\xi_s| \leq 1$ and that simultaneously $\eta$ has exactly $n$ jumps of size in $(2a,a)$ on $(0,\varepsilon)$ and no other jumps in this interval is strictly positive. On the corresponding set we have
		$$2na \re + \varepsilon \gamma_\eta^0/\re \leq \int_0^\varepsilon \re^{-\xi_{s-}} \, \di  \eta_s \leq na\re^{-1} + \re \gamma_\eta^0 \varepsilon.$$
	In particular, by choosing $n$ sufficiently large, we see that for any given $L<0$, the probability that $\int_0^\varepsilon \re^{-\xi_{s-}}\, \di \eta_s<L$ is strictly positive. Using as before that $\int_\varepsilon^\infty \re^{-\xi_{s-}} \, \di s$ is unbounded or a.s. equal to $+\infty$, together with the fact that the probability of $\eta$ having no jumps in the interval $[\varepsilon,\tau]$ is strictly positive, an application of the intermediate value theorem as before shows that $\supp (V_{q,\xi,\eta})  =\bR$ in this case.

\medskip
	Finally, if $\nu_\eta(\bR_+) > 0 = \nu_\eta(\bR_-)$ (as in parts of (iv)) or $\nu(\bR_+), \nu(\bR_-) > 0$ (as in (iii)~(c)), construct $\widehat{\eta}$ from $\eta$ by omitting the positive jumps of $\eta$. Then $\supp (V_{q,\xi,\eta}) \supseteq \supp(V_{q,\xi,\widehat{\eta}})$ by Lemma \ref{support-killing-from-lessjumps} and the claim follows from the cases above.
	
	\medskip
	(v) Assume that $\eta$ is a compound Poisson process with $0\notin\supp(\nu_{\eta})$. Let $T_1,T_2,\dots$ denote the jump times of $\eta$ and $(M_t)_{t\geq0}$ denote the underlying Poisson process, such that~${\eta_t=\sum_{j=1}^{M_t}B_j}$ with the i.i.d. random variables $B_1,B_2,\dots$ following the jump distribution of $\eta$. Consider now the random variables $X_k=\xi_{T_k}-\xi_{T_{k-1}}$ for $k\in\NN$, setting~$T_0=0$. Then $X_1,X_2,\dots$ are i.i.d. and, by the strong Markov property, have the same distribution as $\xi_{T_1}$. Note that $T_1$ is exponentially distributed with parameter $\nu_{\eta}(\RR)$. With probability one it now holds that
	\begin{align*}
		\int_0^{\tau}\re^{-\xi_{s-}}\dx\eta_s&=\int_0^{\tau}\re^{-\xi_s}\dx\eta_s=\sum_{j=1}^{M_{\tau}}\re^{-\xi_{T_j}}B_j=\sum_{j=1}^{M_{\tau}}\re^{-\sum_{k=1}^jX_k}B_j \\ &=\sum_{j=1}^{M_{\tau}}\Big(\prod_{k=1}^j\re^{-X_k}\Big)B_j.
	\end{align*}
	Considering the i.i.d. random variables $A_k=e^{-X_k}$, $k\in\NN$, we can thus write
	$V=\sum_{j=1}^{M_{\tau}}\big(\prod_{k=1}^jA_k\big)B_j$~a.s.,
	where the random variables $A_1,A_2,\dots$ and $B_1,B_2,\dots$are mutually independent.
	Conditioning on $\{M_{\tau}=n\}$ now leads to \eqref{support-eta-cpp}, where $\smash{A_k\overset{d}{=}\re^{-\xi_{T_1}}}$ implies~\eqref{def-Xi}.

\medskip
	Finally, it remains to deduce the explicit form of the set depending on the process $\xi$.
	Writing $\xi_{T_1} = \smash{\int_0^{T_1} \re^0 \di \xi_s}$, we see from the cases (ii)--(iv) already proved that, provided $\xi$ is not a compound Poisson process with $0\not\in \supp (\nu_\xi)$, then
	$\supp(-\xi_{T_1}) =  \{0\}$ if $\xi$ is the zero-process,  $\supp(-\xi_{T_1}) =
	(-\infty, 0]$ if $\xi$ is a non-zero subordinator,  $\supp(-\xi_{T_1}) =
	[0,\infty)$ if $-\xi$ is a non-zero subordinator, and $\supp(-\xi_{T_1}) = \bR$ otherwise.

\medskip
	If $\supp(\xi_{T_1}) = \bR$, it follows immediately from \eqref{support-eta-cpp} that $\supp(V) = [0,\infty)$, $(-\infty,0]$ or~$\bR$ whenever $\eta$ is a subordinator, the negative of a subordinator, or has two-sided jumps, respectively. Similarly, the simplification when $\xi\equiv 0$ is immediate.

\medskip
	Whenever $-\xi$ is a subordinator with $0\in \supp(\nu_\xi)$ or non-zero drift, we clearly have ${\supp(\smash{\re^{-\xi_{T_1}}})=[1,\infty)}$. The corresponding formulas for $\supp(V)$ whenever $\eta$ or $-\eta$ is a subordinator then follow immediately from \eqref{support-eta-cpp}.
	If $\eta$ has jumps of both signs, define $z_+=\inf(\supp(\nu_{\eta})\cap\RR_+)>0$ and $z_-=\sup(\supp(\nu_{\eta})\cap\RR_-)<0$. For  a given number~$x\in\RR$, choose $n\in\NN$ such that ${x-(n-1)z_-\geq z_+}$, as well as
	\begin{equation}\label{choices-ai-bi}
		b_1=\dots=b_{n-1}=z_-,\ a_1=\dots=a_{n-1}=1,\ b_n=z_+,\ a_n=\frac{x-(n-1)z_-}{z_+}.
	\end{equation}
	Then $b_1,\dots,b_n\in \supp(\nu_{\eta})$, $a_n\geq1$ and $\ln(a_1),\dots,\ln(a_n)\in \Xi$ such that we obtain
	$\sum_{j=1}^n \left(\prod_{k=1}^j a_k \right) b_j=x$,
	i.e., $x\in\supp(V)$ by~\eqref{support-eta-cpp} and thus $\supp(V)=\RR$ as claimed.

\medskip
	Finally, assume that $\xi$ is a non-zero subordinator or that it is a compound Poisson process with $\nu_\xi(\bR_+) > 0$. Construct $\widehat{\eta}$ from $\eta$ by deleting all jumps whose absolute size is greater than some constant $K$ but such that $\widehat{\eta}$ still has positive and/or negative jumps if $\eta$ does, and construct $\smash{\widehat{\xi}}$ from $\xi$ by deleting all its negative jumps. Then $V_{0,\widehat{\xi}, \widehat{\eta}}$ exists by \cite[Thm. 2]{EricksonMaller2005}, and by
	\cite[Thm. 1]{BehmeLindnerMaejima2016} its support is $[0,\infty)$, $(-\infty,0]$ or $\bR$, respectively, depending if $\eta$ or $-\eta$ or neither of them is a subordinator. The claim then follows from Lemmas \ref{support-killing-from-nokilling} and \ref{support-killing-from-lessjumps} using
	$\supp(V_{q,\xi,\eta}) \supseteq \supp (V_{q,\widehat{\xi}, \widehat{\eta}}) \supseteq
	\supp (V_{0,\widehat{\xi}, \widehat{\eta}})$.
\end{proof}

\begin{remark}
	{\rm An alternative proof of some of the cases considered in Theorem \ref{thm-support} can be based on a general result regarding the support of solutions of certain random fixed point  equations. Namely, as shown in~\cite[Thm. 2.5.5(1)]{BuraczewskiDamekMikosch2016}, if the random variable~$Z$ is a solution to  the random fixed point equation $\smash{Z\overset{d}{=}AZ+B}$ for two real-valued random variables~$A,B$, where~$Z$ is independent of $(A,B)$, and $A$ and $B$ are such that
		${P(Ax+B=x)<1}$ for every $x\in\RR$, $A\geq0$ a.s., $P(0<A<1)>0$ and $P(A>1)>0$,
		then $\supp(X)$ is either a half-line or $\RR$. In this case it is then enough to determine the left and right end points of the support. Unfortunately, by Theorem \ref{t-sdewithkilling}, this result can be applied in our situation only whenever neither $\xi$ nor $-\xi$ are subordinators. As the proof of these special cases would still have needed a case by case study as given in the present proof of Theorem \ref{thm-support}, we refrained from using the approach via random recurrence equations to keep the proof  of Theorem \ref{thm-support} as short as possible.
	}
\end{remark}

\begin{proof}[Proof of Proposition \ref{thm-support2}]
	(i) First, let $\supp(\nu_{\eta})\subseteq\RR_+$, i.e., $\eta$ is a subordinator. As, by assumption, the probability that the killing occurs before the first jump of $\xi$ is positive, we have~${\supp(\eta_{\tau})\subset\supp(V)}$.
	Recalling the structure of~$\supp(V)$ from Theorem~\ref{thm-support}~(v), $\xi$ being a compound Poisson process implies that
	\begin{displaymath}
			\Xi=\overline{\Big\{\sum_{j=1}^N -x_j, x_j\in\supp(\nu_\xi), N\in \NN_0\Big\}} \subseteq \RR_+.
	\end{displaymath}
	As, by assumption, $[\beta,\alpha]\subseteq\supp(\nu_{\xi})$, it follows directly that~$[-\alpha,-\beta]\subseteq\Xi$ here, as well as~$[-n\alpha,-n\beta]\subseteq\Xi$  for $n\in\NN$. In particular, if we have~${n\geq k=\smash{\lfloor\frac{\alpha}{\beta-\alpha}\rfloor+1}}$, it follows that~${-(n+1)\alpha<-n\beta}$, implying that individual intervals intersect and that, therefore,~$[-k\alpha,\infty)\subseteq\Xi$. Choosing $n=1$, $b_1=\inf\supp(\nu_{\eta})$ and letting $a_1$ run through~$[-k\alpha,\infty)$, we find from~\eqref{support-eta-cpp} that $[\re^{-k\alpha}\inf\supp(\nu_{\eta}),\infty)\subseteq\supp(V)$. If $-\eta$ is a subordinator instead, the claim follows by symmetry.
	
	\medskip
	If $\eta$ has jumps of both signs, denote $z_+\negmedspace=\inf(\supp(\nu_{\eta})\cap\RR_+)$ and\linebreak $z_-\negmedspace=\sup(\supp(\nu_{\eta})\cap\RR_-)$ as in the proof of Theorem~\ref{thm-support} and, for a given $x\in\RR$, choose~$a_1,\dots,a_n,b_1,\dots,b_n$ as in~\eqref{choices-ai-bi}. Then $\ln(a_1),\dots\ln(a_{n-1})\in\Xi$ and also~$\ln(a_n)\in\Xi$ if~$n\in\NN$ is chosen sufficiently large. From~\eqref{support-eta-cpp} we can thus conclude that $x\in\supp(V)$ and, therefore, $\supp(V)=\RR$ as claimed.
	
	\medskip
	(ii) Note that the probability that the killing occurs before the first jump of $\xi$ is positive and so is the probability that, additionally, $\eta$ jumps $n$ times for some $n\in\NN$ in the time interval~$[0,\tau]$. As $\supp(\nu_{\eta})$ contains the interval $[\alpha,\beta]$, it follows that
	\begin{align*}
			\supp(V)&\supseteq\Big\{\sum_{j=1}^n b_j \Big|n\in\NN_0,b_j\in\supp(\nu_{\eta})\Big\}\\
			&\supseteq\Big\{\sum_{j=1}^n b_j \Big|n\in\NN_0,b_j\in[\alpha,\beta]\Big\}=\{0\}\cup\Big(\bigcup_{\ell=1}^{k-1}[\ell \alpha,\ell \beta]\Big)\cup[k\alpha,\infty),
	\end{align*}
	where we set again $k=\smash{\lfloor\frac{\alpha}{\beta-\alpha}\rfloor+1}$. To show that the interval contained in $\supp(V)$ is considerably larger whenever ${\ln (\beta) -\ln(\alpha) \geq - \sup \supp(\nu_\xi)}=:-y_-$ is satisfied, choose $n=1$, as well as $a_1=\smash{e^{-my_-}}$ for~$m\in\NN$ and $b_1\in[\alpha,\beta]$. From~\eqref{support-eta-cpp} it now follows directly that
	$\{\re^{-my_-}[\alpha,\beta],\ m\in\NN\}\subseteq\supp(V)$.
	Since $\ln (\beta) -\ln(\alpha) \geq -y_-$, we find that
	$\re^{-my_-}\alpha\leq \re^{-(m-1)y_-}\beta$
	such that the intervals $\re^{-(m-1)y_-}[\alpha,\beta]$ and $\re^{-my_-}[\alpha,\beta]$ overlap. In particular, we obtain that~${[\alpha,\infty)\cap\supp(V)}=[\alpha,\infty)$, yielding the claim. \\
	Part (iii) follows by symmetry.
	
	\medskip
	(v) Without loss of generality, assume that $z_1=-1$ such that $z_2>0$ is irrational and let~$x\mod1$ denote the quantity $x-\lfloor x\rfloor$ for a real number $x$. Then $z_2\mod1$ is also irrational and the orbit of $z_2$ under the corresponding rotation in the circle group $([0,1),+)$, where $x_1+x_2=(x_1+x_2)\mod1$, is dense (see, e.g.,~\cite[Prop.~1.3.3]{KatokHasselblatt1995}), i.e.,
	\begin{displaymath}
			\overline{\{nz_2\mod1,\ n\in\NN_0\}}=[0,1].
	\end{displaymath}
	Recalling that the killing may occur before the first jump of the process $\xi$, it follows from~\eqref{support-eta-cpp} that
	\begin{align*}
			\supp(V)&\supseteq\overline{\Big\{\sum_{j=1}^nb_n\Big|n\in\NN_0,b_n\in\supp(\nu_{\eta})\Big\}}
			\supseteq\overline{\{n_1z_1+n_2z_2|n_1,n_2\in\NN_0\}}.
	\end{align*}
	Observe that $z_1=-1$ implies $z_2\mod1=z_2+n_0z_1$ for some $n_0\in\NN_0$, yielding that the set on the right-hand side must include the interval $[0,1]$. Further, it includes all translations of $[0,1]$ by numbers $z=\widetilde{n}_1z_1+\widetilde{n}_2z_2$ with fixed $\widetilde{n}_1,\widetilde{n}_2\in\NN_0$, as can be seen from specifying the first $\widetilde{n}_1+\widetilde{n}_2$ jumps. Thus, it follows that $\supp(V)=\RR$. \\
	As the existence of $z_1<0$ and~$z_2>0$ with $\frac{z_2}{z_1}\in\RR\setminus\QQ$ is guaranteed if ${\nu_{\eta}(\RR_-)\neq 0 \neq \nu_{\eta}(\RR_+)}$ and either $\supp(\nu_{\eta})\cap\RR_+$ or $\supp(\nu_{\eta})\cap\RR_-$ contains an interval, Part (iv) follows immediately from the preceding argument.
\end{proof}

\renewcommand*{\bibname}{References}
\bibliographystyle{plain}

\end{document}